\documentclass[10pt, reqno]{amsart}

\usepackage{amsmath,amssymb,amsthm,amsfonts,verbatim}
\usepackage{microtype}
\usepackage[all,2cell]{xy}
\usepackage{mathtools}
\usepackage{graphicx}
\usepackage{pinlabel}
\usepackage{hyperref}
\usepackage{mathrsfs}
\usepackage{color}
\usepackage{enumerate}
\usepackage{cite}

\CompileMatrices

\usepackage[top=1.3in,bottom=1.9in,left=1.2in,right=1.2in]{geometry}

\theoremstyle{plain}
\newtheorem{theorem}{Theorem}[section]
\newtheorem{maintheorem}{Theorem}

\newtheorem{question}[theorem]{Question}

\newtheorem{proposition}[theorem]{Proposition}
\newtheorem{lemma}[theorem]{Lemma}

\newtheorem{fact}[theorem]{Fact}

\newtheorem{claim}[theorem]{Claim}
\newtheorem{addendum}[theorem]{Addendum}

\theoremstyle{definition}
\newtheorem{definition}[theorem]{Definition}
\newtheorem{example}[theorem]{Example}

\newtheorem{remark}[theorem]{Remark}

\newcommand{\nc}{\newcommand}
\nc{\dmo}{\DeclareMathOperator}

\nc{\Q}{\mathbb{Q}}
\nc{\F}{\mathbb{F}}
\nc{\R}{\mathbb{R}}
\nc{\Z}{\mathbb{Z}}
\nc{\C}{\mathbb{C}}
\nc{\Ell}{\mathcal{L}}
\nc{\M}{\mathcal{M}}
\nc{\K}{\mathcal{K}}
\nc{\I}{\mathcal{I}}
\nc{\U}{\mathcal U}
\nc{\disk}{\mathbb{D}}
\nc{\hyp}{\mathbb{H}}

\nc{\CP}{\mathbb{CP}}
\nc{\cS}{\mathcal{S}}
\dmo{\Mod}{Mod}
\dmo{\PMod}{PMod}
\dmo{\LMod}{LMod}
\dmo{\Diff}{Diff}
\dmo{\Homeo}{Homeo}
\dmo{\dist}{dist}
\dmo\BDiff{BDiff}
\dmo\SO{SO}
\dmo\Hom{Hom}
\dmo\SL{SL}
\dmo\Sp{Sp}
\dmo\rank{rank}
\dmo\sig{sig}
\dmo\Out{Out}
\dmo\Aut{Aut}
\dmo\Inn{Inn}
\dmo\GL{GL}
\dmo\PSL{PSL}
\dmo\BHomeo{BHomeo}
\dmo\EHomeo{EHomeo}
\dmo\EDiff{EDiff}
\nc\Sig{\Sigma}
\dmo\Teich{Teich}
\dmo\Fix{Fix}
\nc{\pair}[1]{\langle #1 \rangle}
\nc{\abs}[1]{\left| #1 \right|}
\nc{\action}{\circlearrowright}
\nc{\norm}[1]{\left | \left | #1 \right | \right |}
\nc{\abcd}[4]{\left(\begin{array}{cc} #1 & #2 \\ #3 & #4 \end{array}\right)}
\nc{\into}\hookrightarrow
\dmo{\Isom}{Isom}
\nc{\normal}{\vartriangleleft}
\dmo{\Vol}{Vol}
\dmo{\im}{Im}
\dmo{\Push}{Push}
\dmo{\Conf}{Conf}
\dmo{\PConf}{PConf}
\dmo{\id}{id}
\dmo{\Jac}{Jac}
\dmo{\Pic}{Pic}
\dmo{\Stab}{Stab}
\dmo{\Arf}{Arf}
\dmo{\End}{End}
\dmo{\Gal}{Gal}
\dmo{\lcm}{lcm}
\dmo{\ab}{ab}
\dmo{\opp}{op}
\dmo{\SU}{SU}

\nc{\Span}[1]{\operatorname{Span}(#1)}
\newcommand{\onto}{\twoheadrightarrow}
\newcommand{\ca}{\mathcal}
\newcommand{\lan}{\langle}
\newcommand{\ran}{\rangle}
\newcommand{\ra}{\rightarrow}
\newcommand{\sbs}{\subset}
\newcommand{\ld}{\ldots}
\newcommand{\ov}{\overline}
\newcommand{\Ga}{\Gamma}
\newcommand{\Lam}{\Lambda}
\newcommand{\ti}{\times}
\newcommand{\de}{\delta}
\newcommand{\De}{\Delta}
\newcommand{\Si}{\Sigma}
\newcommand{\al}{\alpha}
\newcommand{\ep}{\epsilon}

\newcommand{\ze}{\zeta}
\newcommand{\ga}{\gamma}
\newcommand{\si}{\sigma}
\newcommand{\hra}{\hookrightarrow}
\newcommand{\xra}{\xrightarrow}
\newcommand{\op}{\oplus}
\newcommand{\scr}{\mathscr}

\newcommand{\wtil}{\widetilde}

\newcommand{\mbf}{\mathbf}
\newcommand{\dra}{\dashrightarrow}

\renewcommand{\epsilon}{\varepsilon}
\renewcommand{\tilde}{\widetilde}
\nc{\coloneq}{\mathrel{\mathop:}\mkern-1.2mu=}
\nc{\margin}[1]{\marginpar{\scriptsize #1}}
\nc{\para}[1]{\medskip\noindent\textbf{#1.}}
\nc{\red}[1]{\textcolor{red}{#1}}

\title{Arithmeticity of the monodromy of some Kodaira fibrations}

\author{Nick Salter \and Bena Tshishiku}
\email{salter@math.harvard.edu \and bena@math.harvard.edu}
\date{May 17, 2018}

\address{Department of Mathematics\\ Harvard University\\ 1 Oxford St., Cambridge, MA 02138}

\begin{document}
\maketitle

\begin{abstract} 
A question of Griffiths-Schmid asks when the monodromy group of an algebraic family of complex varieties is arithmetic. We resolve this in the affirmative for the class of algebraic surfaces known as Atiyah-Kodaira manifolds, which have base and fibers equal to complete algebraic curves. Our methods are topological in nature and involve an analysis of the ``geometric" monodromy, valued in the mapping class group of the fiber.

\end{abstract} 

\section{Introduction}

This paper is focused on certain holomorphic Riemann surface bundles over surfaces commonly known as Atiyah--Kodaira bundles, and whether or not the monodromy group of such a bundle is arithmetic.

Consider a fiber bundle $E\ra B$ with fiber a closed oriented surface $\Si_g$ of genus $g\ge2$. Two important invariants of this bundle are the \emph{monodromy representation} $\mu_E:\pi_1(B)\ra\Mod(\Si_g)$, and the \emph{monodromy group} 
\[\Ga_E=\text{Im}\big[\pi_1(B)\xra{\mu_E}\Mod(\Si_g)\ra\Sp_{2g}(\Z)\big].\]

The group $\Ga_E<\Sp_{2g}$ is called \emph{arithmetic} if it is finite index in the $\Z$-points of its Zariski closure; otherwise $\Ga_E$ is called \emph{thin}. It is a poorly understood problem when a family of algebraic varieties has arithmetic monodromy group and which arithmetic groups arise as monodromy groups. For more information, see \cite[\S1]{venkataramana}. 

Given a surface $X$ of genus $g_0\ge2$ and $m\ge2$, Atiyah and Kodaira independently constructed holomorphic Riemann surface bundles 
\[E^{nn}(X,m)\ra B,\] where $B$ is a closed surface, and the fiber is a certain cyclic branched cover $Z\ra X$ (the number $m$ describes the local model $z\mapsto z^m$ of the cover over the branched points). Denote $g=\text{genus}(Z)$ and write $\Ga^{nn}(X,m)=\Ga_{E^{nn}(X,m)}$ for the monodromy group in $\Sp_{2g}(\Z)$. The superscript ``$nn$" stands for ``non-normal," which will be explained shortly. 

\begin{maintheorem}\label{theorem:main-nn}
Fix $m\ge2$, and let $X$ be a closed surface of genus $g_0\ge5$. Then the monodromy group $\Ga^{nn}(X,m)$ of the Atiyah--Kodaira bundle $E^{nn}(X,m)\ra B$ is arithmetic. 
\end{maintheorem}


In the course of studying $\Ga^{nn}(X,m)$, we will determine the Zariski closure $\mathbf G^{nn}$ of $\Ga^{nn}(X,m)$. There is an obvious candidate. By the nature of the construction, the fiber $Z$ carries an action of $Q\simeq\Z/m\Z$, and $\Ga^{nn}(X,m)$ acts on $H_1(Z)$ by $\Q[Q]$-module maps, and so preserves the decomposition $H_1(Z;\Q)=\bigoplus_{k\mid m}N_k$ into isotopic factors for the simple $\Q[Q]$-modules. Then 
\begin{equation}\label{equation:Z}\Ga^{nn}(X,m)<\prod_{k\mid m}\Aut_Q(N_k,\pair{\cdot,\cdot}_Q)<\Sp_{2g}(\Q),\end{equation}
where $\pair{\cdot,\cdot}_Q:N_k\times N_k\ra\Q(\ze_k)$ is the \emph{Reidemeister pairing} induced from the intersection pairing $(\cdot,\cdot)$ on $H_1(Z)$ (see Section \ref{section:monodromy}). 

We'll denote $\mbf G_k:=\Aut_Q(N_k,\pair{\cdot,\cdot}_Q)$. Another artifact of the construction of $E^{nn}(X,m)\ra B$ is that the projection of $\Ga^{nn}$ to $\mbf G_1$ is trivial. Then the obvious candidate (or at least an obvious ``upper bound") for the Zariksi closure of $\Ga^{nn}(X,m)$ is the group $\prod_{k\mid m,\> k\neq 1}\mathbf G_k$. We show this is the Zariski closure if and only if $m$ is prime.

\begin{maintheorem}\label{theorem:zariski-nn}
Fix notation as in Theorem \ref{theorem:main-nn}. Let $\mathbf G^{nn}$ be the Zariski closure of $\Ga^{nn}(X,m)$. 
\begin{enumerate}
\item if $m$ is prime, then $\mbf G^{nn}=\mathbf G_m$; 
\item if $m$ is composite, then $\mathbf G^{nn}$ is a proper subgroup of $\prod_{k\mid m,\>k\neq1}\mathbf G_k$. 
\end{enumerate} 
\end{maintheorem}

In the composite case, the precise description of $\mathbf G^{nn}$ is more complicated but can be determined. See Section \ref{section:nonnormal}.

\begin{remark}[Arithmetic quotients of mapping class groups]
Theorem \ref{theorem:main-nn} is in the spirit of and builds off work of Looijenga \cite{looijenga}, Grunewald--Larsen--Lubotzky--Malestein \cite{GLLM}, Venkataramana \cite{venkataramana}, and McMullen \cite{mcmullen} that we now describe. Given a finite, regular (possibly branched) cover $\Si_g\ra\Si'$ with deck group $G$, there is a \emph{virtual homomorphism} $\rho:\Mod(\Si')\dra\Mod(\Si_g)^G\ra\Sp_{2g}(\Z)^G$ to the centralizer of $G<\Sp_{2g}(\Z)$. Here ``virtual" means that the homomorphism is only defined on a finite-index subgroup of $\Mod(\Si')$. If the cover $\Si\ra\Si'$ is unbranched, then under mild assumptions \cite{looijenga} and \cite{GLLM} showed that $\rho$ is \emph{almost onto}, i.e.\ the image is a finite index subgroup. Venkataramana \cite{venkataramana} proved a similar theorem for certain branched covers of the disk. Using the same techniques of the proof of Theorem \ref{theorem:main-nn} (or Theorem \ref{theorem:main} below) in combination with an analysis of how powers of Dehn twists lift (Section \ref{section:AKmonodromy1}), we have the following analogue of Venkataramana's theorem for branched covers of higher-genus surfaces. 


\begin{addendum}\label{corollary:mod-quotient}
Let $Z\ra X$ be a branched cover of surfaces of the kind described in Section \ref{subsection:nn}. If the genus of $X$ is at least $5$, then the virtual homomorphism $\Mod(X,x)\dra\Mod(Z)^Q\ra\Sp_{2g}(\Z)^Q$ is almost onto. The same conclusion holds for the virtual homomorphism associated to the branched covers $W\ra X$ considered in Section \ref{subsection:repair}. 
\end{addendum}

We remark that Addendum \ref{corollary:mod-quotient} is in some sense easier than Theorem \ref{theorem:main-nn} since in the Corollary we have the full flexibility of a finite-index subgroup of $\Mod(Y,\mbf y)$, whereas in the Theorem, we are constrained to an infinite-index surface subgroup $\pi_1(B)<\Mod(Y,\mbf y)$. However, the idea is the same for both, and this article owes an intellectual debt to the works of Looijenga, Grunewald--Larsen--Lubotzky--Malestein, and Venkataramana cited above.
\end{remark}

\begin{remark}[Multiple fiberings of surface bundles]
One of the remarkable features of the Atiyah--Kodaira bundles is that they admit two distinct surface bundle structures. For example, if $g_0=2$ and $m=2$, then the bundle $E^{nn}(X,m)\ra B$ has base of genus 129 and fiber of genus 6, and the total space also fibers $E^{nn}(X,m)\ra B'$, where $B'$ has genus 3 and the fiber has genus 321. It is natural to ask if there are any other fiberings; see e.g. \cite[Question 3.4]{salter-fiberings} . Combining our computation with work of L.\ Chen \cite{chen-fiberings}, we are able to answer this question. 

\begin{addendum}\label{corollary:multiple-fiberings}
The total space $E^{nn}(X,m)$ of the Atiyah--Kodaira bundle fibers in exactly two ways. 
\end{addendum} 

The proof closely parallels the argument of \cite{chen-fiberings}, and so for brevity's sake we content ourselves with a brief sketch. We need only supply a version of \cite[Lemma 3.4]{chen-fiberings} applicable to any $E^{nn}(X,m)$. This follows easily from Theorem \ref{theorem:main-nn} (and the description of the Zariski closure $\mbf G^{nn}$ of $\Ga^{nn}(X,m)$; see Section \ref{section:nonnormal}). 
\end{remark}

\begin{remark}[Surface group representations and rigidity]
The real points $\mathbf G^{nn}(\R)$ of the Zariski closure of $\Ga^{nn}(X,m)$ is frequently a Hermitian Lie group (e.g.\ this is always true when $m$ is prime). Thus the Atiyah--Kodaira bundles provide a naturally-occurring family of surface group representations into Hermitian Lie groups with arithmetic image. As such they are potentially of interest in higher Teichm\"uller theory. In this regard, we remark that Ben Simon--Burger--Hartnick--Iozzi--Wienhard \cite{weakly-maximal} introduced a notion of \emph{weakly maximal} surface group representations into Hermitian Lie groups based on properties of their Toledo invariants in bounded cohomology. The Atiyah--Kodaira surface group representations have nonzero Toledo invariants (it is described in \cite[\S5.3]{tshishiku} how to compute them), but these representations are not weakly maximal because they are not injective. 
\end{remark}

\begin{remark}[Kodaira fibrations and the Griffiths-Schmid problem]
The Atiyah--Kodaira bundles fit into a larger class of examples known as \emph{Kodaira fibrations}. A Kodaira fibration is a holomorphic map $f: E \to B$ where $E$ is a complex algebraic surface and $B$ is a closed Riemann surface, such that $f$ is the projection map for a {\em differentiable}, but not a {\em holomorphic} fiber bundle. There are many variants and extensions of the branched-cover construction method; see, e.g. \cite{catanese-survey}, but there is essentially only one other known method for constructing Kodaira fibrations. This proceeds by the ``Satake compactification'' $\overline{\mathcal M_g^{sat}}$ of the moduli space $\mathcal M_g$ (the compactification induced from the Satake compactification $\mathcal A_g \subset \overline{\mathcal A_g}$ of the moduli space of principally-polarized Abelian varieties). For $g \ge 3$, the boundary of $\overline{\mathcal M_g^{sat}}$ has (complex) codimension at least 2. It follows that a generic iterated hyperplane section of $\overline{\mathcal M_g^{sat}}$ in fact lies in $\mathcal M_g$, producing a complete algebraic curve $C$ embedded in $\mathcal M_g$, and hence (after passing to a suitable finite-sheeted cover $C' \to C$), a Kodaira fibration $E \to C'$. See \cite[Section 1.2.1]{catanese-survey} for details. 

It is easy to see that these Kodaira fibrations have arithmetic monodromy groups, e.g. by an appeal to a suitable version of the Lefschetz hyperplane theorem as applied to $C \le \overline{\mathcal A_g}$. This prompts the following question.
\begin{question}\label{question:arithmetic?}
Does every Kodaira fibration have arithmetic monodromy group? 
\end{question}
Question \ref{question:arithmetic?} is a version of the famous problem of Griffiths-Schmid \cite[page 123]{griffiths-schmid}, who pose the question of arithmeticity of monodromy groups for any smoothly-varying family of algebraic varieties. The work of Deligne-Mostow \cite{deligne-mostow} furnishes certain examples of families of algebraic curves over higher-dimensional, quasi-projective bases $B$, for which the monodromy group is shown to be non-arithmetic. Question \ref{question:arithmetic?} is motivated by the authors' curiosity as to whether the restriction to the class of Kodaira fibrations (where the bases are required to be {\em projective} curves) imposes enough rigidity to enforce arithmeticity.
\end{remark}

\para{About the proof of Theorem \ref{theorem:main-nn}} To prove Theorem \ref{theorem:main-nn} we need to (i) identify the Zariski closure $\mbf G^{nn}$ and (ii) prove $\Ga^{nn}(X,m)$ is finite index in $\mbf G^{nn}(\Z)$. We identify $\mbf G^{nn}$ in two steps, which one can view as an ``upper" and ``lower" bound, and arithmeticity of $\Ga^{nn}(X,m)$ will be an immediate consequence:
\begin{enumerate}
\item[(a)] Upper bound. Identify a $\Q$-subgroup $\mbf G'<\Sp_{2g}(\Q)$ so that $\Ga^{nn}(X,m)<\mbf G'$. This implies that $\mbf G^{nn}\le\mbf G'$. 
\item[(b)] Lower bound. Show that $\Ga^{nn}(X,m)$ contains ``enough" (see Proposition \ref{proposition:UU}) unipotent elements to generate a finite-index subgroup of $\mbf G'(\Z)$. This implies that $\mbf G'\le\mbf G^{nn}$. 
\end{enumerate} 
Together (a) and (b) imply that $\mbf G^{nn}=\mbf G'$. Then (b) also implies $\Ga^{nn}(X,m)$ is arithmetic.

As mentioned in (\ref{equation:Z}), there is an obvious upper bound, but according to Theorem \ref{theorem:zariski-nn}, this upper bound is frequently not sharp. This is related to the fact that the cover $Z\ra X$ is not normal, which causes major technical difficulties in understanding $\Ga^{nn}(X,m)$ directly by the above scheme. To bypass these difficulties, we consider a further cover $W\ra Z$ for which $W\ra X$ is normal. Specifically, we obtain a diagram 
\[\begin{xy}
(-10,0)*+{E(X,m)}="A";
(15,0)*+{E^{nn}(X,m)}="B";
(-10,-12)*+{B'}="C";
(15,-12)*+{B}="D";
(-10,12)*+{W}="E";
(15,12)*+{Z}="F";
{\ar"A";"B"}?*!/_3mm/{};
{\ar "A";"C"}?*!/^5mm/{};
{\ar "B";"D"}?*!/_3mm/{};
{\ar"C";"D"}?*!/_3mm/{};
{\ar "E";"A"}?*!/_3mm/{};
{\ar"F";"B"}?*!/_3mm/{};
\end{xy}\]
where each vertical sequence is a fibration, and the horizontal maps are covering maps. The bundle $E(X,m)\ra B'$ has a monodromy group $\Ga(X,m)$, and there is a commutative diagram 
\[\begin{xy}
(-10,0)*+{\pi_1(B')}="A";
(12,0)*+{\Ga(X,m)}="B";
(-10,-12)*+{\pi_1(B)}="C";
(12,-12)*+{\Ga^{nn}(X,m)}="D";
{\ar@{->>}"A";"B"}?*!/_3mm/{};
{\ar@{^{(}->} "A";"C"}?*!/^5mm/{};
{\ar "B";"D"}?*!/_3mm/{};
{\ar@{->>}"C";"D"}?*!/_3mm/{};
\end{xy}\]
It is easy to see that the image of $\Ga(X,m)$ in $\Ga^{nn}(X,m)$ is of finite index. Our approach will be to first determine the Zariski closure of $\Ga(X,m)$ and show $\Ga(X,m)$ is arithmetic (via the strategy outlined above) and then relate this back to $\Ga^{nn}(X,m)$. 

\para{\emph{The upper bound}} Let $\mbf G$ be the Zariski closure of $\Ga(X,m)$. Whereas the fiber $Z$ of $E^{nn}(X,m)\ra B$ has an action of $Q\simeq\Z/m\Z$, the fiber $W$ of $E(X,m)\ra B'$ has an action of the Heisenberg group $H\simeq\mathscr H(\Z/m\Z)$ (see \eqref{equation:Hpres}).
As before, there is an obvious upper bound on $\mbf G$ that comes from considering the decomposition $H_1(W;\Q)=\bigoplus M_{k,\chi}$ of $H_1(W;\Q)$ as a $\Q[H]$ module, where the sum is indexed by the simple $\Q[H]$ modules (see Section \ref{section:heisenberg}). 
Then as before, 
\begin{equation}\label{equation:W}\Ga(X,m)<\prod\Aut_H(M_{k,\chi},\pair{\cdot,\cdot}_H)<\Sp(H_1(W)).\end{equation}
We'll denote $\mbf G_{k,\chi}:=\Aut_H(M_{k,\chi},\pair{\cdot,\cdot}_H)$. When $k=1$, the projection of $\Ga(X,m)$ to $\mbf G_{k,\chi}$ is trivial, so 
\begin{equation}\label{equation:upperW}\mbf G\le \prod_{k\mid m,\>k\neq 1}\prod_\chi \mbf G_{k,\chi}.\end{equation}

\para{\emph{The lower bound}} In this case we are able to produce enough unipotent elements in $\Ga(X,m)$ to show that (\ref{equation:upperW}) is an equality: 

\begin{maintheorem}\label{theorem:main}
Fix $m\ge 2$ and let $X$ be be a surface of genus $g_0\ge5$. The monodromy group $\Ga(X,m)$ of the normalized Atiyah--Kodaira bundle $E(X,m)\ra B'$ is arithmetic. It has Zariski closure $\mbf G\simeq \prod_{k\mid m,\>k\neq 1}\prod_\chi \mbf G_{k,\chi}$. 
\end{maintheorem}

We briefly remark on how Theorem \ref{theorem:main} is proved. The fibers of $E(X,m)\ra B'$ admit an action of $H$, and so we can consider the bundle $E(X,m)/H\ra B'$, which is a bundle with fiber $X$. The monodromy of this bundle is well-understood: it is easily describable in terms of ``point-pushing'' diffeomorphisms $P(\ga)$ on $X$. Theorem \ref{theorem:main} is proved by (i) understanding when $P(\ga)$ lifts to $W$ and how it acts on $H_1(W)$, and (ii) finding many elements $P(\ga)$ whose action on $H_1(W)$ is unipotent. This latter part is the main technical aspect of the paper.

\para{Section outline} The paper is roughly divided into sections as follows: 
\begin{itemize}
\item Sections \ref{section:AKconstruction} and \ref{section:monodromy}: topology of covering spaces. We recall the Atiyah--Kodaira construction and give a new variation that we call the \emph{normalized Atiyah--Kodaira construction}; we give topological models for the surfaces $X,Z,W$ that appear as fibers in these constructions and compute explicit generators for the homology $H_1(\cdot)$ of these surfaces as modules over various deck groups; we recall the Reidemeister pairing on the homology of a regular cover and give a mapping-class-group description of how the monodromy changes under fiberwise covers.
\item Sections \ref{section:AKmonodromy1} and \ref{section:AKmonodromy2}: mapping class group computations. To lay the groundwork for studying $\Ga(X,m)$, we examine when Dehn twists and point-pushing diffeomorphisms of a surface $\Si'$ lift to a branched cover $\Si\ra\Si'$ and how lifts act on $H_1(\Si)$. Using this analysis, we proceed to find many unipotent elements in $\Ga(X,m)$ and prove Theorem \ref{theorem:main}. 
\item Sections \ref{section:unipotents}, \ref{section:heisenberg}, \ref{section:nonnormal}: representation theory and algebraic groups. In Section \ref{section:unipotents} we recall some results about generating an arithmetic group by unipotent elements that we will use to give the aforementioned ``lower bound" on the Zariski closure of $\Ga(X,m)$. In Section \ref{section:heisenberg} we discuss the representation theory over $\Q$ for the Heisenberg group $\mathscr H(\Z/m\Z)$. In Section \ref{section:nonnormal} we compare the algebraic groups appearing in (\ref{equation:Z}) and (\ref{equation:W}), and combine this with Theorem \ref{theorem:main} to prove Theorems \ref{theorem:main-nn} and \ref{theorem:zariski-nn}. 
\end{itemize}

\para{Acknowledgements} The authors thank B.\ Farb, from whom they learned about this problem. The authors also thank J.\ Malestein for suggesting how to get information about the ``non-normalized" monodromy group from the ``normalized" monodromy group.  The authors acknowledge support from NSF grants DMS-1703181 and DMS-1502794, respectively.

\section{Atiyah--Kodaira manifolds}\label{section:AKconstruction}
\subsection{The Atiyah--Kodaira construction, globally}\label{subsection:nn}

The bundles under study in this paper are a refinement of a construction first investigated by Kodaira \cite{kodaira}. Shortly thereafter, Atiyah independently developed the same construction \cite{atiyah}, and so this class of examples is known as the ``Atiyah--Kodaira construction''. Our treatment in this paragraph follows the presentation in \cite[Section 4.3]{moritabook}.

Fix a positive integer $m > 1$. Let $X$ be a compact Riemann surface of genus $g_0 \ge 2$. Let $p: Y \to X$ be a cyclic unbranched covering with deck group $\pair{\sigma} \cong \Z/m\Z$. For $0 \le i \le m-1$, define the locus
\[
\Gamma_i := \{(y, \sigma^iy)\mid y \in Y\} \subset Y \times Y.
\]
Let $p': B \to Y$ be the unbranched regular covering corresponding to the homomorphism
\[
\pi_1(Y) \onto H_1(Y; \Z/m\Z).
\]
Define $\wtil\Ga\subset B\times Y$ as the preimage of $\bigcup_{i=0}^{m-1}\Ga_i$ under $p'\times\id:B\times Y\ra Y\times Y$. 
Since $p: Y \to X$ is a regular unbranched covering, $\widetilde \Gamma$ intersects each fiber $\{b\} \times Y$ in exactly $m$ distinct points. An analysis involving the K\"unneth formulas for $B \times Y$ and $Y \times Y$ (see \cite[Section 4.3]{moritabook} for details) shows that
\begin{equation}\label{equation:mdiv}
\left[\widetilde \Gamma\right] = 0 \mbox{ in } H_2(B \times Y; \Z/m\Z).
\end{equation}
Viewing $\widetilde \Gamma$ as a divisor on $B \times Y$, (\ref{equation:mdiv}) implies that there is a line bundle $L \in \Pic(B \times Y)$ such that $mL = [\widetilde \Gamma]$. Let $\pi: E(L) \to B \times Y$ denote the projection from the total space of $L$. There is a map $f$ of line bundles 
\[\begin{xy}
(-10,0)*+{E(L)}="A";
(15,0)*+{E([\wtil\Ga])}="B";
(2.5,-12)*+{B\times Y}="C";
{\ar"A";"B"}?*!/_3mm/{f};
{\ar "A";"C"}?*!/^5mm/{};
{\ar "B";"C"}?*!/_3mm/{};
{\ar@/_/ "C";"B"}?*!/^2mm/{s};
\end{xy}\]
Here $s$ is the section with divisor $\widetilde \Gamma$, and the restriction of $f$ to each fiber has model $z\mapsto z^m$.  The Atiyah--Kodaira construction is the algebraic surface 
\[
E^{nn}(X,m) := (f^{-1}\circ s)(B\times Y);
\]
the superscript $nn$ stands for ``non-normal'' and will be explained in the following paragraph. By construction, this is an $m$-fold cyclic branched covering of $B \times Y$ branched along $\widetilde \Gamma$. As remarked above, $\widetilde \Gamma$ intersects each fiber $\{y_0\} \times Y$ in exactly $m$ distinct points. Restricted to some such fiber, the branched covering $E^{nn}(X,m) \to B \times Y$ restricts to an $m$-fold cyclic branched covering $q: Z \to Y$ branched at $m$ points. Denote the covering group by $\pair{\zeta} \cong \Z/m\Z$. The projection $\pi: E^{nn}(X,m) \to B$ endows $E^{nn}(X,m)$ with the structure of a Riemann surface bundle over $B$ with fibers diffeomorphic to $Z$. 

\subsection{Repairing normality}\label{subsection:repair} In the remainder of the section, we will undertake a study of the Atiyah--Kodaira construction within the setting of the theory of surface bundles. Of primary importance will be a ``fiberwise'' description of the construction outlined above.

 By construction, $q: Z \to Y$ is an $m$-fold cyclic branched covering, and $p: Y \to X$ is an $m$-fold cyclic unbranched covering. Let $Z^\circ$ denote the subsurface of $Z$ on which $q$ restricts to an unbranched covering, and define $Y^\circ = q(Z^\circ)$. By construction $Z^\circ$ is $Z$ with $m$ points removed, and $Y^\circ$ is similarly $Y$ with $m$ points removed. Moreover, the $m$ removed points in $Y^\circ$ correspond to the $m$ points of intersection of $Y$ with the divisor $\tilde \Gamma$, and by construction this is the set $\{(y, \sigma^i y)\}$ for some fixed $y \in Y$ and $1 \le i \le m$. It follows that $p: Y \to X$ restricts to an unbranched covering $p: Y^\circ \to X^\circ$. By the above discussion, $X^\circ$ is $X$ with the single point $x = p(y) = p(\sigma^i y)$ removed.

The coverings $q: Z^\circ \to Y^\circ$ and $p: Y^\circ \to X^\circ$ are regular by construction. However, we will see below that the composite $p \circ q: Z^\circ \to X^\circ$ is {\em not} regular. This presents serious difficulties for the study of the monodromy of the bundle $E^{nn}(X,m) \to B$. To repair this, we will pass to a further (unbranched) cover $r: W \to Z$, such that the composite $p \circ q \circ r: W \to X$ becomes a regular (albeit non-abelian) cover.

To describe $W$, it is helpful to make a more explicit study of $X,Y,Z$. Let $X$ be a surface of genus $g_0$, represented as a $2g_0$-gon $\Delta$ with edges $\{e_1, f_1, \dots, e_{g_0}, f_{g_0}\}$ identified so that the word around $\partial \Delta$ (traversed counterclockwise) reads $e_1f_1e_1^{-1}f_1^{-1} \dots e_{g_0}f_{g_0}e_{g_0}^{-1}f_{g_0}^{-1}$. Let $E_i$ be the oriented curve on $X$ represented on $\Delta$ as a segment connecting the edge labeled $e_i$ to the edge $e_i^{-1}$, and define $F_i$ analogously (see Figure \ref{figure:YtoX}). The curves $\{E_i, F_i\}$ furnish a set of geometric representatives for a basis of $H_1(X; \Z)$. Via the intersection pairing $(\cdot, \cdot)$, this also leads to a basis for $H^1(X;\Z)$. Explicitly, a class $v \in H_1(X;\Z)$ determines the element $(v,\cdot) \in H^1(X;\Z)$. 

The cover $Y \to X$ is regular with deck group $\pair{\sigma} \cong \Z/m\Z$. Such covers are classified by elements of $H^1(X, \Z/m\Z)$. Relative to the basis for $H^1(X; \Z)$ given above, we take the cover $Y \to X$ to correspond to the element $(F_1, \cdot) \pmod m$. There is an explicit model for $Y$ as a union of $m$ copies of a $2g_0$-gon. For $1 \le i \le m$, let $\Delta_i$ be a copy of the labeled $2g_0$-gon above. Identify $e_1$ on $\Delta_i$ with $e_1^{-1}$ on $\Delta_{i+1}$ (interpreting subscripts mod $m$), and identify all other edges $\eta$ on $\Delta_i$ with their counterpart $\eta^{-1}$ on $\Delta_i$. See Figure \ref{figure:YtoX}.

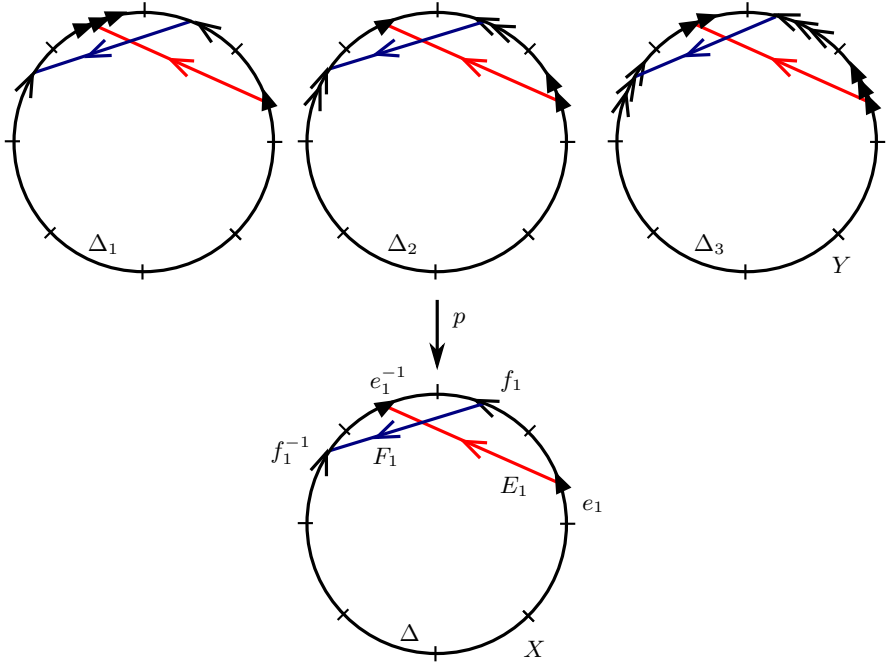
\begin{figure}
\labellist
\small
\pinlabel $\Delta$ [bl] at 147 6
\pinlabel $\Delta_1$ [bl] at 29.6 152
\pinlabel $\Delta_2$ [bl] at 142.4 152
\pinlabel $\Delta_3$ [bl] at 258 152
\pinlabel $E_1$ [tr] at 200 72
\pinlabel $F_1$ [t] at 144 84
\pinlabel $e_1$ [l] at 215 59
\pinlabel $f_1$ [bl] at 185 100
\pinlabel $e_1^{-1}$ [b] at 145 98.4
\pinlabel $f_1^{-1}$ [br] at 118 72.8
\pinlabel $p$ [l] at 166.4 129
\pinlabel $X$ [c] at 200 5
\pinlabel $Y$ [tl] at 310 155
\endlabellist
\includegraphics{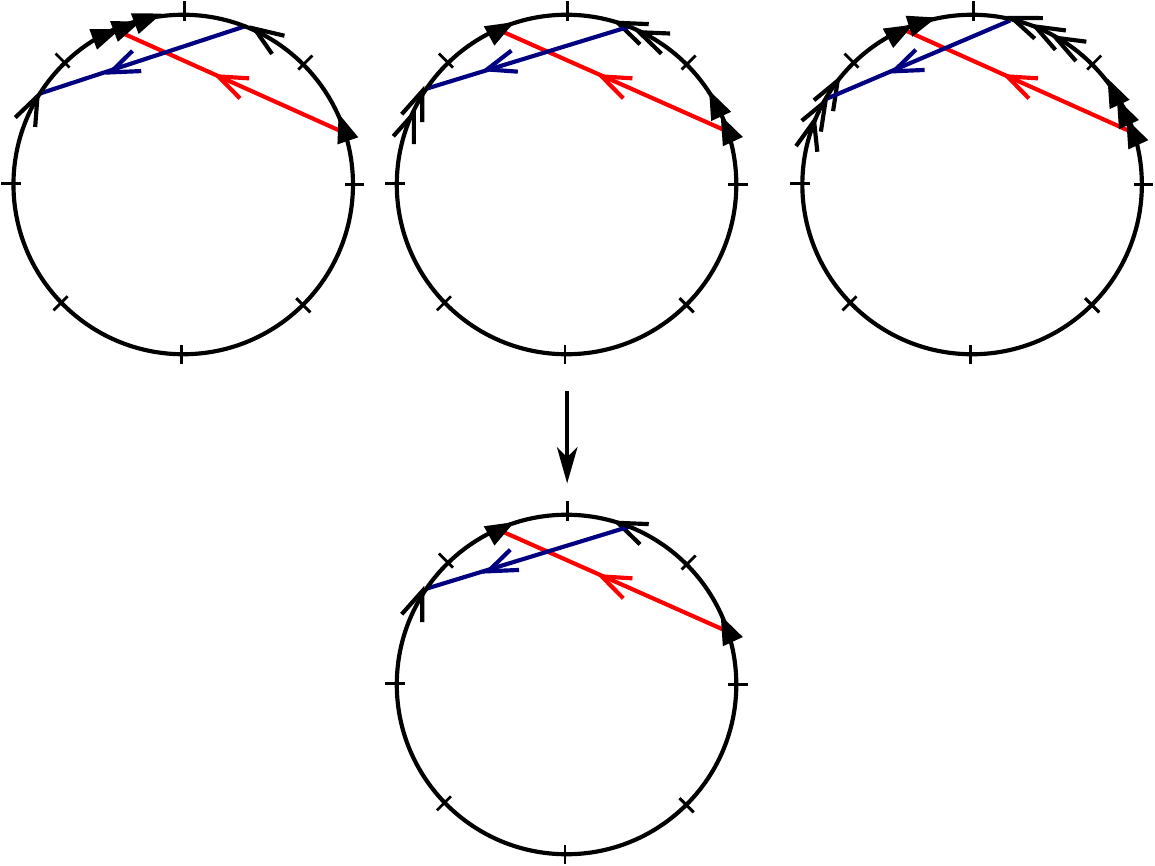}
\caption{The covering $p: Y \to X$, illustrated for $m = 3$ and $g_0 = 2$. Here and throughout, we suppress edge identifications whenever confusion is unlikely.}
\label{figure:YtoX}
\end{figure}
	
$H_1(Y;\Z)$ has the structure of a $\Z[\Z/m\Z]$-module which can be described explicitly as follows.
\begin{lemma}\label{lemma:H1Y}
Identify $\Z/m\Z \cong \pair{\sigma}$. Then there is an isomorphism 
\[
H_1(Y;\Z) \cong \Z[\sigma]^{2 g_0 - 2} \oplus \Z^2
\]
of $\Z[\sigma]$-modules. Explicitly, there are generators $\{E_i,F_i \mid 1 \le i \le g_0\} \subset H_1(Y;\Z)$ such that 
\[
\Z[\sigma]\pair{E_1, F_1} \cong \Z^2
\]
and
\[
\Z[\sigma]\pair{E_i, F_j \mid 2 \le i,j \le g_0} \cong \Z[\sigma]^{2g_0-2}.
\]
\end{lemma}
\begin{proof}
Let $p: Y \to X$ denote the projection. For $2 \le i \le m$, the preimage $p^{-1}(E_i)$ consists of $m$ components, and the same is true for $p^{-1}(F_j)$ for $1 \le j \le m$. By abuse of notation, we define the curves $E_i$ and $F_j$ as the component of the appropriate preimage that is contained in the polygon $\Delta_1$. The preimage $p^{-1}(E_1)$ has a single component, which we denote simply by $E_1$, continuing to abuse notation. The proof now follows by inspection.
\end{proof}

The covering $q: Z \to Y$ is a $\Z/m\Z$ {\em branched} covering with branch locus $L= \{\sigma^i(y)\}$ for some $y \in Y$. As above, set $Y^\circ:= Y \setminus L$. The covering $q$ is classified by some element $\theta \in H^1(Y^\circ; \Z/m\Z)$. The inclusion $Y^\circ \into Y$ induces the short exact sequence
\[
1 \to K \to H_1(Y^\circ; \Z) \to H_1(Y; \Z) \to 1.
\]
The kernel $K$ can be described explicitly as follows. Assume $y \in Y$ is chosen so as to lie in the interior of $\Delta_1$. Let $C$ be a small loop encircling $y$. Then 
\begin{equation}\label{equation:K}
K \cong \Z\pair{\sigma^i C \mid 1 \le i \le m} / \{C + \sigma C +\dots + \sigma^{m-1} C = 0\}.
\end{equation}

Assume $y \in Y$ has been chosen so as to be disjoint from the curves $\sigma^i E_j$ and $\sigma^i F_j$ on $Y$. Then the collection of $\sigma^i E_j, \sigma^i F_j$ determines a splitting 
\begin{equation}\label{equation:decomp}
H_1(Y^\circ; \Z) \cong H_1(Y; \Z) \oplus K.
\end{equation}
Relative to this splitting, the class $\theta \in H^1(Y^\circ; \Z/m\Z)$ that classifies the branched cover $q: Z \to Y$ is defined so that $\theta(\sigma^i C) = 1$ and $\theta \equiv 0$ on $H_1(Y; \Z)$. As $\theta$ is valued in $\Z/m\Z$, this determines a well-defined class on $K$.

The cover $q: Z \to Y$ can be described explicitly by using branch cuts. For $1 \le j \le m-1$, let $\gamma_j$ be the oriented arc beginning at $\sigma^j y \in \Delta_j$ that crosses $e_1$ onto $\Delta_{j+1}$ and ends at $\sigma^{j+1}y \in \Delta_{j+1}$. Take $m$ copies of $Y \setminus \bigcup \{\gamma_j\}$, labeled $Y_1,\dots, Y_m$. To construct $Z$, glue the right side of $\gamma_j$ on sheet $Y_i$ to the left side of $\gamma_j$ on sheet $Y_{i+j}$ (as usual, interpret all subscripts mod $m$). The covering group of $q: Z\to Y$ is isomorphic to $\Z/m\Z$; let $\zeta$ be a generator. It is straightforward to check that this construction really does determine the cover determined by $\theta$. See Figure \ref{figure:ZtoY}. Note that in this figure, the points deleted in passing to $Z^\circ$ (and $Y^\circ$) are depicted by the small circles at the center of each polygon. Hence Figure \ref{figure:ZtoY} is {\em also} a depiction of the unbranched covering $q: Z^\circ \to Y^\circ$. 

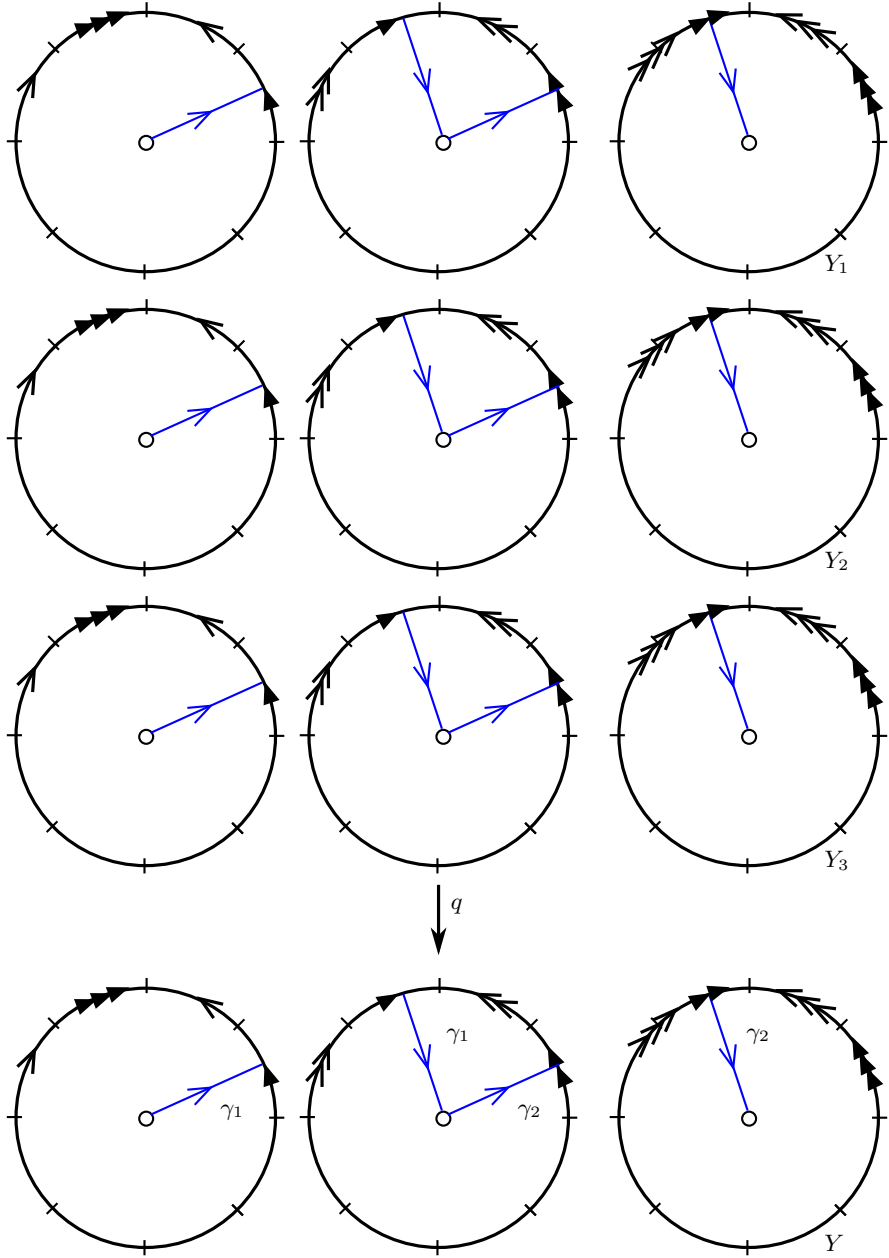
\begin{figure}
\labellist
\small
\pinlabel $\gamma_1$ [br] at 92 49.6
\pinlabel $\gamma_2$ [br] at 204 49.6
\pinlabel $\gamma_1$ [tl] at 164 88
\pinlabel $\gamma_2$ [tl] at 277 88
\pinlabel $Y$ [tl] at 306.4 10.4
\pinlabel $Y_3$ [tl] at 306.4 155
\pinlabel $Y_2$ [tl] at 306.4 268
\pinlabel $Y_1$ [tl] at 306.4 380
\pinlabel $q$ [bl] at 165.6 127.2
\endlabellist
\includegraphics{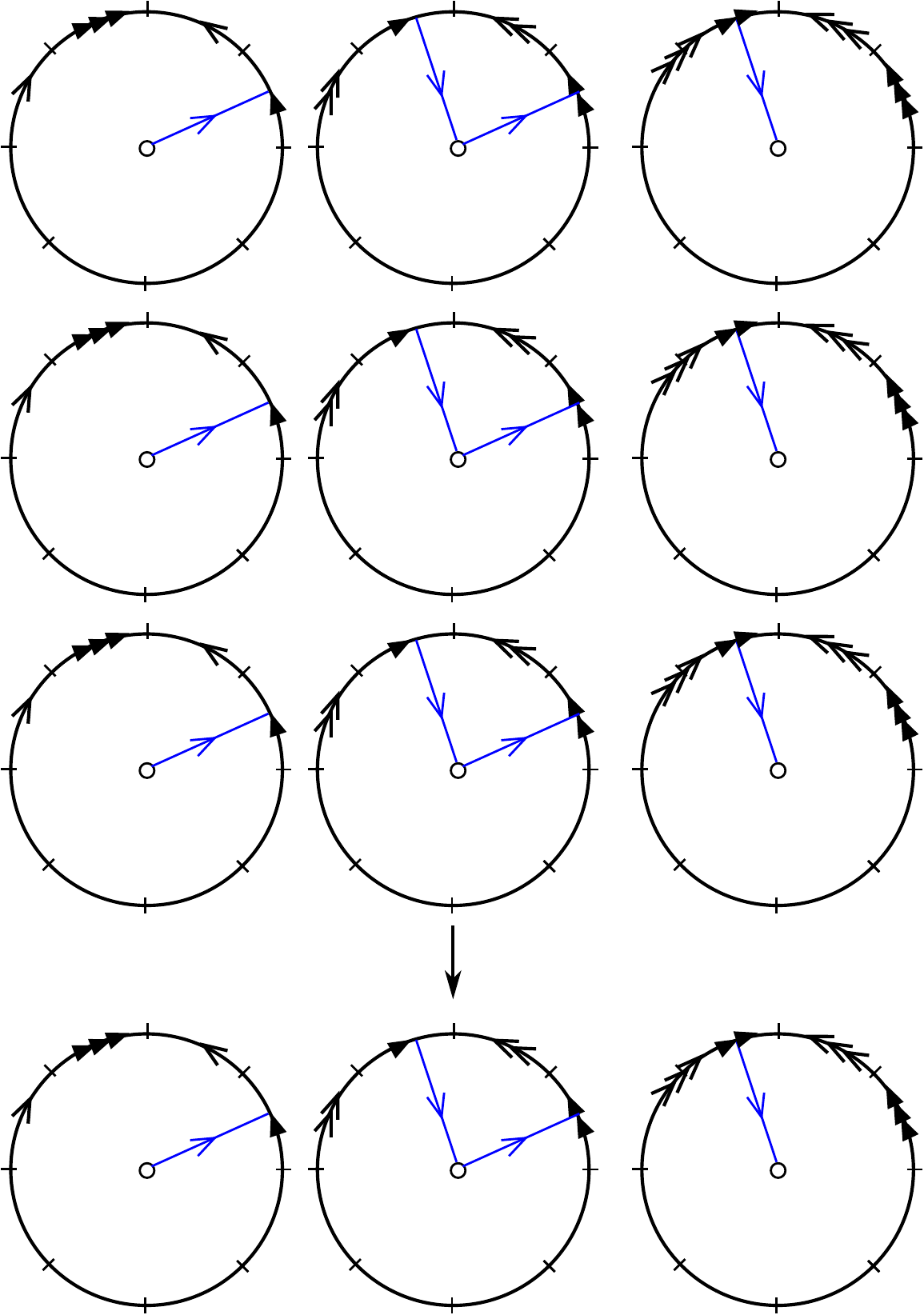}
\caption{The covering $q: Z \to Y$, illustrated for $m = 3$ and $g_0 = 2$.}
\label{figure:ZtoY}
\end{figure}
	
Having fixed this model for $Z$, one sees explicitly the non-regularity of the covering $p\circ q: Z^\circ \to X^\circ$. Consider the curve $F_1 \subset X^\circ$. Then $p^{-1}(F_1) \subset Y^\circ$ has $m$ components, one on each polygon $\Delta_i$. The component contained in $\Delta_i$ is denoted $F_{1,i}$. One sees that $q^{-1}(F_{1,1})$ has $m$ components, while $q^{-1}(F_{1,2})$ has one component. This prevents the $\si$-action on $Y$ from lifting to $Z$, and so $Z\ra X$ is not a normal cover. Despite this, one can repair the regularity by passing to a further cyclic cover.

\begin{lemma}\label{lemma:WZnormal}
Let $r: W \to Z$ be the cyclic unbranched covering classified by the element $\alpha \in H^1(Z; \Z/m\Z)$ defined as
\[
\alpha= ((p\circ q)^{-1}(E_1),\cdot).
\]
Let the covering group for $r: W \to Z$ be denoted $\pair{\tau} \cong \Z/m\Z$. Then 
\[
p \circ q \circ r: W^\circ \to X^\circ
\]
is a regular covering with $m^3$ sheets. Moreover, the covering group $H$ admits an explicit presentation via
\begin{equation}\label{equation:Hpres}
H \cong \pair{\sigma, \tau, \zeta \mid [\sigma, \tau] = \zeta,\ \zeta \mbox{ central}, \si^m=\tau^m=\ze^m=1} = \mathscr{H}(\Z/m\Z).
\end{equation}
Here $\mathscr{H}(\Z/m\Z)$ denotes the Heisenberg group over $\Z/m\Z$. 
\end{lemma}

\begin{proof}
We will first define an auxiliary covering $\pi: U^\circ \to X^\circ$ which is regular with covering group $H$ by construction; then we will exhibit an isomorphism $U^\circ \cong W^\circ$ of covers of $X^\circ$. 

We describe $U^\circ \to X^\circ$ in terms of a homomorphism $h: \pi_1(X^\circ) \to H$. The fundamental group $\pi_1(X^\circ)$ is a free group of rank $2g_0$ and admits a presentation of the form
\begin{equation}\label{equation:pi1pres}
\pi_1(X^\circ) = \pair{e_1,f_1, \dots, e_{g_0}, f_{g_0}, c\mid [e_1,f_1] \dots [e_{g_0},f_{g_0}] = c}.
\end{equation}
Geometrically, the elements $e_i$ (resp. $f_i$) correspond to loops crossing the edge $E_i$ (resp. $F_i$) of $\Delta$, and the element $c$ corresponds to a loop that encircles the deleted point in $X^\circ$ counterclockwise. Define
\[
h: \pi_1(X^\circ) \to H
\]
via $h(e_1) = \sigma$, $h(f_1) = \tau$, $h(c) = \zeta$ with all other generators mapped to the identity $1 \in H$. It is immediate from the presentations (\ref{equation:Hpres}) and (\ref{equation:pi1pres}) that $h$ is well-defined. 

The coverings $\pi: U^\circ \to X^\circ$ and $p\circ q\circ r: W^\circ \to X^\circ$ correspond to subgroups $\pi_1(U^\circ), \pi_1(W^\circ)$ of $\pi_1(X^\circ)$. To show that $U^\circ \cong W^\circ$ are isomorphic as covers of $X^\circ$, it suffices to show that $\pi_1(U^\circ) = \pi_1(W^\circ)$ as subgroups of $\pi_1(X^\circ)$. To this end, define
\[
h': \pi_1(X^\circ) \to H
\]
by $h'(e_1) = \sigma$, with all other generators sent to $1 \in H$. It is clear that $\ker(h') = \pi_1(Y^\circ)$. It is elementary to verify that $\pi_1(Y^\circ)$ admits a presentation with generators
\[
\{e_1^m, f_1\} \cup \{e_1^i e_j e_1^{-i}, e_1^i f_j e_1^{-i}\mid 0 \le i \le i-1,\ 2 \le j \le g_0 \} \cup \{e_1^i c e_1^{-i} \mid 0 \le i \le m-1\},
\]
and a single relation that expresses $\prod_{i = 1}^m e_1^{m-i} c e_1^{i-m}$ as a product of commutators of the remaining generators.

It follows that the map 
\[
h'': \pi_1(Y^\circ) \to H
\]
for which
\[
h''(f_1) = \tau, \quad h''(e_1^i c e_1^{-i}) = \zeta
\]
(and all other generators sent to $1 \in H$) is well-defined. A comparison with the explicit description of the {\em regular} covering $q \circ r: W^\circ \to Y^\circ$ shows that $\ker(h'') = \pi_1(W^\circ)$. On the other hand, there is a description of $H$ as a semi-direct product
\[
H \cong \pair{\tau, \zeta} \rtimes \pair{\sigma}.
\]
From this, one sees that $\ker(h'') = \ker(h) = \pi_1(U^\circ)$. The result follows. 
\end{proof}

\subsection{The normalized Atiyah--Kodaira construction} Above we gave a global construction of the manifold $E^{nn}(X,m)$. In this paragraph we describe a finite cover of this space that we call the {\em normalized Atiyah--Kodaira construction}.

\begin{proposition}[Normalized Atiyah--Kodaira construction]\label{proposition:normalized}
Let $E^{nn}(X,m)$ be an Atiyah--Kodaira manifold that fibers over $B$ with fiber $Z$. There is an unbranched cover $R: E(X,m) \to E^{nn}(X,m)$ with the following properties:
\begin{itemize}
\item $R$ is a regular covering with deck group $\Z/m\Z$.
\item $E(X,m)$ is the total space of a $W$-bundle over a surface $B'$.
\item The surface $B'$ is a finite unbranched cover of $B$.
\item Let $E^{nn}(X,m)'$ denote the pullback of the bundle $E^{nn}(X,m)$ along the cover $B' \to B$. Then the map $R$ factors $E(X,m)\xra{R'} E^{nn}(X,m)'\ra E^{nn}(X,m)$, where $R'$ is a bundle map that covers $\id: B'\to B'$. Fiberwise, the restriction $R'\mid_{W}: W \to Z$ is the unbranched covering $r: W \to Z$ described above.
\end{itemize}
\end{proposition}
\begin{proof}
As remarked above, unbranched $\Z/m\Z$-coverings of a topological space $S$ are classified by $H^1(S; \Z/m\Z)$. The claims of the proposition will follow from the construction of an element $\tilde \alpha \in H^1(E^{nn}(X,m)'; \Z/m\Z)$ such that the pullback of $\tilde \alpha$ to $H^1(Z; \Z/m\Z)$ is the element $\alpha$ of Lemma \ref{lemma:WZnormal}. 

We first observe that the branching locus $\tilde \Gamma$ of $E^{nn}(X,m)$ is a disjoint union of $m$ sections $B\ra B\times Y$, so $E^{nn}(X,m)\to B$ admits a section. 
Consequently, the $5$-term exact sequence for $E^{nn}(X,m) \to B$ degenerates, yielding a splitting
\[
H^1(E^{nn}(X,m); \Z/m\Z) \cong H^1(B; \Z/m\Z) \oplus H^1(Z; \Z/m\Z)^{\pi_1(B)}.
\]
As $H^1(Z; \Z/m\Z)$ is finite, there is some finite-index subgroup $\pi_1(B') \le \pi_1(B)$ such that $\alpha \in H^1(Z; \Z/m\Z)$ is $\pi_1(B')$-invariant. Define $E^{nn}(X,m)'$ to be the pullback of $E^{nn}(X,m)$ along the cover $B' \to B$. Then the $5$-term sequence for $E^{nn}(X,m)' \to B'$ shows that there exists a class $\tilde \alpha$ with the required properties. 
\end{proof}

\subsection{The homology of $W$} We will need to understand $H_1(W;\Z)$ and $H_1(W;\Q)$, especially as representations of the covering group $H = \mathscr{H}(\Z/m\Z)$. Rational coefficients will be assumed unless otherwise specified. For our purposes we will require an explicit set of generators for $H_1(W)$ as a $\Q[H]$-module. We remark that if we were only interested in the character of $H_1(W)$ as an $H$-representation, then we could obtain this indirectly using the Chevalley--Weil theorem (see Lemma \ref{lemma:Mfree}). 

Our description of $H_1(W)$ will be derived in two steps. Let $s: V \to X$ denote the unbranched $(\Z/m\Z)^2$-covering associated to the homomorphism 
\begin{equation}\label{equation:hV}
h: \pi_1(X) \to H/\pair{\zeta} \cong (\Z/m\Z)^2
\end{equation}
given by $h(e_1) = \sigma$ and $h(f_1) = \tau$.

\begin{lemma}\label{lemma:H1V}
Identify $\Q[(\Z/m\Z)^2] \cong \Q[\sigma,\tau]$. There is an isomorphism 
\[
H_1(V) \cong \Q[\sigma, \tau]^{2 g_0 - 2} \oplus \Q^2
\]
of $\Q[\sigma, \tau]$-modules. Explicitly, there are generators $\{E_i,F_i \mid 1 \le i \le g_0\} \subset H_1(V)$ such that 
\[
\Q[\sigma, \tau]\pair{E_1, F_1} \cong \Q^2
\]
and
\[
\Q[\sigma,\tau]\pair{E_i, F_j \mid 2 \le i,j \le g_0} \cong \Q[\sigma, \tau]^{2g_0-2}.
\]
\end{lemma}
\begin{proof}
Essentially the same as in Lemma \ref{lemma:H1Y}. There is an explicit model for $V$ built out of $m^2$ copies of the polygon $\Delta$, indexed as $\Delta_{i,j}$. The symbols $E_k, F_k$ correspond to the homology classes of the component of $s^{-1}(E_k)$ (resp. $s^{-1}(F_k)$) contained in $\Delta_{1,1}$. The preimages $s^{-1}(E_1)$ and $s^{-1}(F_1)$ each have $m$ components, and the remaining $s^{-1}(E_k), s^{-1}(F_k)$ for $k \ge 2$ each contain $m^2$ components. It is readily verified that (i) the components of $s^{-1}(E_1)$ (resp. $s^{-1}(F_1)$) are mutually homologous, and that (ii) the collection of components of $s^{-1}(E_k), s^{-1}(F_k)$ for $k \ge 2$ span a $\Q$-subspace of rank $m^2(2g_0-2)$ transverse to the span of $E_1, F_1$. The claims follow from these observations.
\end{proof}

The surface $W$ arises as a $\Z/m\Z$ {\em branched} covering $t: W \to V$. By construction, the covering $s\circ t: W \to X$ coincides with $p \circ q \circ r$. We write $\Q[H] = \Q[\sigma, \tau, \zeta]$ with the understanding that $\sigma, \tau, \zeta$ are subject to the relations in $H$. Under this identification, the cover $t: W \to V$ has deck group $\pair{\zeta}$.

The homology of the branched cover requires a more delicate analysis than in the unbranched case, and will require some preliminary ideas.

\begin{definition}[Planar form]
Let $f: \Si \to \Si'$ be a regular branched covering of Riemann surfaces with deck group $\Z/m\Z$. Let $L\sbs \Si'$ denote the branching locus. Then $f$ is said to be in {\em planar form relative to $D$} if there is a disk $D \subset \Si'$ such that $L \subset D$, and such that $f^{-1}(\Si' \setminus D)$ is a disjoint union of $m$ copies of $\Si' \setminus D$. 
\end{definition}

\begin{definition}[$G$-curve]\label{definition:Gcurve}
Let $f: \Si  \to \Si'$ be a branched covering in planar form relative to $D$. Let $\gamma \subset D$ be an arc connecting distinct points $p_1, p_2 \in L$, such that $\gamma$ is disjoint from all other elements of $L$. Then an associated {\em $G$-curve}, written $G_\gamma$, is one of the $m$ curves consisting of the two copies of $\gamma$ on adjacent sheets of $f^{-1}(D)$. (Sheets $D_1, D_2$ are ``adjacent'' if a loop encircling $p_1$ starting on $D_1$ has endpoint on $D_2$.) 
\end{definition}
\begin{remark} Note that the different choices for $G_\gamma$ are all equivalent under the action of the deck group for $f: \Si  \to \Si'$.
\end{remark}

When $f:\Si \to \Si'$ is in planar form, $H_1(\Si )$ has a simple description in terms of $H_1(\Si')$ and a system of $G$-curves.
\begin{lemma}\label{lemma:brcov}
Let $f: \Si  \to \Si'$ be a $\Z/m\Z$-branched covering in planar form relative to $D$; identify $\Z/m\Z \cong \pair{\zeta}$. Let $\{\gamma_i\}$ be a collection of arcs as in Definition \ref{definition:Gcurve} such that $\{[\gamma_i]\}$ generates $H_1(D, L)$. 
Then there is a surjective map of $\Q[\zeta]$-modules
\[
g:\Q[\zeta] H_1(\Si') \oplus \Q[\zeta] \pair{\{G_{\gamma_i}\}} \to H_1(\Si ).
\]
Moreover, $g$ is injective when restricted to $\Q[\zeta] H_1(\Si')$. 
\end{lemma}
\begin{proof}
As $f$ is in planar form, the Mayer-Vietoris sequence provides an exact sequence
\[
H_1(f^{-1}(\partial D)) \to  H_1(f^{-1}(\Si' \setminus D)) \oplus H_1(f^{-1}(D)) \to H_1(\Si ) \to 1.
\]
Again since $f$ is in planar form, $f^{-1}(\Si' \setminus D)$ consists of $m$ disjoint copies of $\Si' \setminus D$, acted on in the obvious way by the deck group $\pair{\zeta}$. Thus $H_1(f^{-1}(\Si' \setminus D)) \cong \Q[\zeta] H_1(\Si')$. It remains to be seen that $H_1(f^{-1}(D))$ is generated as a $\Q[\zeta]$-module by $\{[G_{\gamma_i}]\}$ under the assumption that $\{[\gamma_i]\}$ generate $H_1(D,L)$. 

Let $D^\circ$ denote the disk $D$ with a small neighborhood of each branch point $p_i \in L$ for $1 \le i \le k$ removed. Thus $D^\circ$ is a sphere with $k+1$ boundary components. We describe a cell structure on $D^\circ$. The zero-skeleton is given by 
\[
(D^\circ)^{(0)} = \{v_1, \dots, v_k, w\},
\]
with each $v_i$ on the boundary component associated to $p_i$, and $w \in \partial D$.  Next, take 
\[
(D^\circ)^{(1)} = \{c_1, \dots, c_k, e_1, \dots, e_k, d\},
\]
with both ends of $c_i$ attached to $v_i$, each $e_i$ connecting $v_i$ and $w$, and both ends of $d$ attached to $w$. Then $(D^\circ)^{(2)}$ consists of a single $2$-cell attached in the obvious way. 

The  above cell structure lifts to a $\Q[\zeta]$-equivariant cell structure on $f^{-1}(D^\circ)$. The boundary maps on the $1$-cells are $\Q[\zeta]$-linear and are given by
\[
\partial(c_i) = (\zeta - 1) c_i, \quad \partial(e_i) = v_i - w, \quad \partial(d) = 0.
\]
{\em A priori}, one knows that $H_0(f^{-1}(D^\circ)) = \Q$. Thus, $\partial: C_1(f^{-1}(D^\circ)) \to C_0(f^{-1}(D^\circ))$ has corank $1$ as a map of $\Q$-vector spaces. A dimension count then shows that $Z_1(f^{-1}(D^\circ))$ has dimension $mk+1$ over $\Q$. An argument in elementary linear algebra then implies that $Z_1(f^{-1}(D^\circ))$, and hence $H_1(f^{-1}(D^\circ))$, is generated over $\Q[\zeta]$ by the set
\[
\{(\sum_{j=1}^m \zeta^j) c_i \mid 1 \le i \le k\} \cup \{(\zeta-1)(e_1 - e_i) + (c_i - c_1) \mid 2 \le i \le k\} \cup \{d\}.
\]

Topologically, the inclusion map $f^{-1}(D^\circ) \to f^{-1}(D)$ attaches $k$ disks along the boundary components encircling the branch points $p_i$. The boundary of these disks are represented by the classes $\{(\sum_{j=1}^m \zeta^j) c_i \mid 1 \le i \le k\}$. It follows that $H_1(f^{-1}(D))$ is generated over $\Q[\zeta]$ by the set
\[
\{(\zeta-1)(e_1 - e_i) + (c_i - c_1) \mid 2 \le i \le k\} \cup \{d\}.
\]
The summand spanned by $[d]$ clearly corresponds to $H^1(f^{-1}(\partial D))$. The result will now follow from the description of a surjective map 
\[
\pi: \Q[\zeta]H_1(D,L) \to  \Q[\zeta]\pair{(\zeta-1)(e_1 - e_i) + (c_i - c_1) \mid 2 \le i \le k} \le H_1(f^{-1}(D)).
\]
such that $\pi(\gamma) = G_\gamma$ for all arcs $\gamma$ connecting two points of $L$.

Using the cell structure on $D^\circ$ described above (which can be extended to a cell structure on $D$ by adding $k$ additional $2$-cells), the set $L$ is identified with the set $\{v_1, \dots, v_k\}$. Then a generating set for $H_1(D,L)$ consists of the $k-1$ elements $e_1 - e_i$ for $2 \le i \le k$. Define $\pi$ by setting
\[
\pi(e_1 - e_i) = (\zeta-1)(e_1 - e_i) + (c_i - c_1).
\]
It is evident from the construction that $\pi(e_1 - e_i) = G_{e_1-e_i}$ and that $[G_{\gamma_1 + \gamma_2}] = [G_{\gamma_1}] + [G_{\gamma_2}]$. The result now follows by linearity. 
\end{proof}

We now apply Lemma \ref{lemma:brcov} to the branched covering $t: W \to V$. Continuing to abuse notation, we let $E_i\subset W$ denote a single component of the preimage $(s \circ t)^{-1}(E_i)$, and define $F_i$ similarly. Recalling the construction of $X$ in terms of the polygon $\Delta$, we observe that $V$ can be constructed from $m^2$ copies of $\Delta$ indexed by elements $(i,j) \in (\Z/m\Z)^2$. Each $\Delta_{i,j}$ has a marked point $p_{i,j}$ corresponding to the unique branch point for $t$ contained in $\Delta_{i,j}$. We define $\gamma_h$ to be the arc starting at $p_{0,0}$ that crosses $e_1$ onto $\Delta_{1,0}$ and ends at $p_{1,0}$. Likewise, $\gamma_v$ is defined to be the arc starting at $p_{0,0}$ that crosses $f_1$ onto $\Delta_{0,1}$ and ends at $p_{0,1}$. Then the curves $G_h$ and $G_v$ on $W$ are defined by
\begin{equation}\label{equation:Ghdef}
G_h:= G_{\gamma_h}, \quad G_v:= G_{\gamma_v}
\end{equation}
in the sense of Definition \ref{definition:Gcurve}. 

\begin{lemma}\label{lemma:H1W}
The simple closed curves
\[
\{E_i, F_i \mid 1 \le i \le g_0\} \cup \{G_h, G_v\}
\]
generate $H_1(W)$ as a $\Q[H]$-module. Moreover, the submodule spanned by $E_i,F_i$ for $2 \le i \le g_0$ determines a free $\Q[H]$-module of rank $2g_0 - 2$. 
\end{lemma}
\begin{proof}
In order to apply Lemma \ref{lemma:brcov}, it is necessary to identify a disk $D$ relative to which $t: W \to V$ is in planar form. It is straightforward to verify that a small regular neighborhood of the collection of arcs
\[
\{\sigma^i \tau^j \gamma_h \mid 0 \le i\le m-2,\ 0 \le j \le m-1\} \cup \{\tau^j \gamma_v \mid 0 \le j \le m-2\}
\]
is a disk $D \subset V$ containing all points $p_{i,j}$. Moreover, the components of $s^{-1}(E_1), s^{-1}(F_1)$ passing through $\Delta_{0,0}$ are  disjoint from $D$, and the same is true for the entire preimage $s^{-1}(E_j), s^{-1}(F_j)$ for $j \ge 2$. In particular, these curves generate $H_1(V)$ over $\Q[\sigma,\tau]$. 

Applying Lemma \ref{lemma:brcov}, we now see that $H_1(W)$ is generated as a $\Q[H]$ module by $\{E_i, F_i \mid 1 \le i \le g_0\}$, along with a certain subset of curves in the $\Q[H]$-orbit of $G_h$ and $G_v$. This proves the first claim. The second claim follows from the second assertion of Lemma \ref{lemma:brcov} in combination with Lemma \ref{lemma:H1V}. 
\end{proof}

Define the element
\[
\Pi_{na} = (m - (1 + \zeta + \dots + \zeta^{m-1})) \in \Z[H];
\]
over $\Q$, this determines the projection onto the summand of $\Q[H]$ spanned by irreducible $H$-representations that do not factor through the abelianization $H^{\ab}$. (We call such representations {\em non-abelian}, c.f.\ Section \ref{section:heisenberg}.)

\begin{lemma}\label{lemma:Mfree}
For $d = 2g_0 - 1$, there is an isomorphism
\[
\Pi_{na}H_1(W;\Q) \cong \Pi_{na} \Q[H]^d
\]
of $\Pi_{na}\Q[H]$-modules. 
\end{lemma}

\begin{proof}
This follows readily from the determination of the $\Q[H]$-module structure of $H_1(W; \Q)$ via the method of Chevalley-Weil. This in turn follows from an elaboration of the method of Lemma \ref{lemma:brcov}. Place a cell structure on $X^\circ$ as follows: there are two zero-cells $v,w$; $2g_0+2$ one-cells $a_1, b_1, \dots, a_{g_0},b_{g_0}, c, e$; and one two-cell $F$. Each cell $a_i, b_i$ has both ends attached to $v$, the cell $c$ has both ends attached to $w$, and $e$ connects $v$ to $w$. The two-cell $F$ is attached in the obvious way (this does not need to be described in detail). 

This gives rise to the chain complex $C_\bullet(X^\circ)$ computing $H_1(X^\circ; \Q)$:
\[
\Q\pair{F} \to \Q\pair{a_1, \dots, b_{g_0}, c, e} \to \Q\pair{v,w} \to 0.
\]
Lifting this cell structure along the covering map $W^\circ \to X^\circ$, we arrive at an $H$-equivariant cell structure on $W^\circ$. On the level of chain complexes, 
\[
C_\bullet(W^\circ) = \Q[H] \otimes_{\Q} C_{\bullet}(X^\circ).
\]

We wish to determine the character $\chi(H_1(W^\circ; \Q))$. This can be obtained by taking the Euler characteristic of the chain complex $C_\bullet(W^\circ)$, viewed as a virtual character of $H$. Since every $\Q[H]$ module is semisimple, 
\[
\chi(C_\bullet(W^\circ)) = \chi(H_\bullet(W^\circ;\Q)).
\]
This provides the following equality of characters:
\[
\chi(C_0(W^\circ))- \chi(C_1(W^\circ)) + \chi(C_2(W^\circ)) = \chi(H_0(W^\circ;\Q))- \chi(H_1(W^\circ;\Q)) + \chi(H_2(W^\circ;\Q)).
\]
By construction, each $C_i(W^\circ)$ is a free $\Q[H]$-module on $2, 2g_0+2, 1$ generators, respectively. On the right-hand side, we observe that $H_0(W^\circ;\Q) \cong \Q$ and $H_2(W^\circ; \Q) = 0$. Altogether, this determines the character of $H_1(W^\circ;\Q)$ completely and furnishes an isomorphism
\[
H_1(W^\circ; \Q) \cong \Q[H]^{2g_0-1} \oplus \Q.
\]

To determine $H_1(W;\Q)$ as a $\Q[H]$-representation, we exploit the $H$-equivariant exact sequence
\begin{equation}\label{equation:HSES}
1 \to K \to H_1(W^\circ; \Q) \to H_1(W; \Q) \to 1
\end{equation}
induced by the inclusion map $W^\circ \to W$. The kernel $K$ is the subspace of $H_1(W^\circ; \Q)$ spanned by loops around the punctures of $W^\circ$. The cover $W \to X$ factors through the intermediate cover $V$, and the covering $V \to X$ is unbranched. Consequently, the punctures of $W^\circ$ are in one-to-one correspondence with the elements of the covering group $H^{\ab} \cong (\Z/m\Z)^2$, and this bijection intertwines the action of the deck group $H$ with multiplication by $H^{\ab}$.

On the level of $H_1(W^\circ;\Q)$, this implies that $K$ is spanned by the $H$-orbit of a single puncture $[c]$, and that $[c]$ is stabilized by $[H,H] = \pair{\zeta}$. The $\Q$-span of $h[c]$ for $h \in H^{\ab}$ is subject to the single relation 
\[
\sum_{h \in H^{\ab}}h[c] = 0.
\]
In other words, there is an isomorphism of $\Q[H]$-modules
\[
K \cong \Q[H^{\ab}]/ \Q[1];
\]
here $\Q[1]$ denotes the trivial submodule. 

Taking the Euler characteristics of the short exact sequence (\ref{equation:HSES}), we determine the character of $H_1(W; \Q)$ and find that
\begin{equation}\label{equation:H1WSES}
H_1(W; \Q) \cong \Q^2 \oplus \Q[H]^{2g_0 - 2} \oplus \Q[H]/\Q[H^{\ab}].
\end{equation}

Finally, to determine $\Pi_{na}H_1(W;\Q)$ as a $\Pi_{na} \Q[H]$-module, we multiply both sides of (\ref{equation:H1WSES}) by $\Pi_{na}$ and obtain
\[
\Pi_{na} H_1(W; \Q) \cong \Pi_{na}\Q[H]^{2g_0-1}.\qedhere
\]
\end{proof}

Below, we record some properties of the intersection form $(\cdot, \cdot)$ on $H_1(W)$ in the basis specified by Lemma \ref{lemma:H1W}. 

\begin{lemma}\label{lemma:intformW}
Let $N_1 \le H_1(W)$ denote the $\Q[H]$-submodule spanned by the set $\{E_i,F_i \mid 1 \le i \le g_0\}$, and let $N_2$ denote the $\Q[H]$-submodule spanned by $G_h, G_v$. 
\begin{enumerate}
\item $N_1$ and $N_2$ are orthogonal with respect to $(\cdot, \cdot)$.
\item Let $\xi, \eta \in H$ and $C_1, C_2 \in \{E_i,F_i \mid 1 \le i \le g_0\}$ be given. Then
\[
(\xi C_1, \eta C_2) = \begin{cases}\pm1 & \mbox{ if } \xi = \eta \mbox{ and } \{C_1,C_2\} = \{E_i,  F_i\} \mbox{ for some } i,\\
				0 & \mbox{ otherwise}.
\end{cases}
\]
\end{enumerate}
\end{lemma}
\begin{proof}
This follows immediately from the explicit construction of $W$ described above. The various orthogonality relations above are all consequences of the {\em disjointness} of the curves representing the homology classes in question.
\end{proof}

\section{Monodromy of surface bundles}\label{section:monodromy}

\para{The monodromy group} Theorem \ref{theorem:main} is concerned with arithmetic properties of the monodromy of the bundle $E(X,m) \to B'$. In this paragraph we define the monodromy groups in question. Throughout $E\ra B$ will denote an arbitrary bundle (in particular, the base space $B$ has nothing to do with the base of the Atiyah--Kodaira bundle). For a proof of Proposition \ref{proposition:mexistence} below (as well as general background on surface bundles), see \cite[Section 5.6.1]{FM}. 

\begin{proposition}[Existence of monodromy representation]\label{proposition:mexistence}
Let $B$ be any paracompact Hausdorff space and $\Sigma_g$ a closed oriented surface of genus $g \ge 2$. Associated to any $\Sigma_g$-bundle $\pi: E \to B$ is a homomorphism called the {\em monodromy representation}
\[
\mu: \pi_1(B,b) \to \Mod(\Sigma_g),
\]
well-defined up to conjugacy. Informally, $\mu$ records how a local identification of the fiber $\pi^{-1}(b) \cong \Sigma_g$ changes as the fiber is transported around loops in $B$. 
\end{proposition}

While understanding this mapping class group-valued monodromy will be essential in the ensuing analysis, our ultimate goal is to understand an algebraic ``approximation'' to $\mu$. 

\begin{definition}[The symplectic representation]\label{definition:symplectic}
Let 
\[
\Psi: \Mod(\Sigma_g) \to \Aut(H_1(\Sigma_g; \Z), (\cdot,\cdot))
\]
be the homomorphism induced by the action of $\Homeo(\Sigma_g)$ on $H_1(\Sigma_g;\Z)$. 
The notation $\Aut(H_1(\Sigma_g; \Z), (\cdot,\cdot))$ indicates the group of automorphisms of $H_1(\Sigma_g;\Z)$ preserving the algebraic intersection pairing $(\cdot, \cdot)$. This group is isomorphic to the symplectic group $\Sp(2g,\Z)$. For $E \to B$ a $\Sigma_g$ bundle with monodromy representation $\mu$, the \emph{symplectic representation} is the composition  
\[
\rho := \Psi \circ \mu. 
\]
\end{definition}

For the remainder of the paper, we fix the notation $\mu_{X,m}$ for the monodromy of $E(X,m)$, as well as
\[
\hat\Gamma_{X,m} := \im(\mu_{X,m})
\]
and
\[
\Gamma_{X,m} := \im(\rho_{X,m}).
\]
Often $X$ and/or $m$ will be implicit and we will write simply $\hat \Gamma, \Gamma$.

\para{Fiberwise coverings and monodromy}
As described above, the Atiyah--Kodaira bundle is constructed as a fiberwise branched covering. In this paragraph we establish some basic facts concerning the structure of monodromy representations of such bundles. Throughout this paragraph, we fix the following setup: let $\Si \to \Si'$ be a regular covering of Riemann surfaces with finite deck group $G$, possibly branched. Suppose that $\pi: E \to B$ is a $\Si$-bundle, $\pi': E' \to B$ is a $\Si'$-bundle, and there is a fiberwise branched covering $E \to E'$. The monodromy representations for $ \pi:  E \to B$ and $\pi': E' \to B$ will be denoted $\mu, \mu'$, respectively. 

Proofs for Lemmas \ref{lemma:Ginv} -- \ref{lemma:rhotarget} to follow can be found in \cite{looijenga} and \cite{GLLM}. 

\begin{lemma}\label{lemma:Ginv}
Let $\Mod(\Si)^G\le \Mod(\Si)$ denote the centralizer of the subgroup $G \le \Mod(\Si)$. Then there is a finite cover $B' \to B$ such that $\mu \mid_{\pi_1(B')}$ has image in $\Mod(\Si)^G$. 
\end{lemma}

The action of $G$ on $\Si$ endows $H_1(\Si;\Q)$ with the structure of a $\Q[G]$-module. Since $G$ is finite, $H_1(\Si;\Q)$ is a semisimple module, and thus there is a decomposition
\begin{equation}\label{equation:CW}
H_1(\Si;\Q) \cong \bigoplus_{U_i} U_i^{m_i}
\end{equation}
where the sum runs over the simple $\Q[G]$-modules $U_i$. The summands $U_i^{m_i}$ are known as {\em isotypic factors}. The classical Chevalley-Weil theorem gives a complete description of each multiplicity $m_i$, as long as one has a complete list of the simple $\Q[G]$-modules (or their characters). The latter can be worked out in theory, but can be tedious in practice. In Sections \ref{section:AKmonodromy1} and \ref{section:AKmonodromy2} 
we will not need to know the $m_i$ (nor even the $U_i$) explicitly. We'll see that the mere {\em existence} of the decomposition (\ref{equation:CW}) has consequences for the study of the monodromy $\mu$. (Later in Section \ref{section:nonnormal} we will need to know something about the decomposition for $G$ the Heisenberg group -- we establish the necessary facts in Section \ref{section:heisenberg}). 

In light of Lemma \ref{lemma:Ginv}, in the remainder of the paragraph we will assume that $\mu$ is valued in $\Mod(\Si)^G$. The following shows that such an assumption has strong consequences for the symplectic monodromy representation $\rho : \pi_1(B) \to \Aut(H_1(\Si), (\cdot, \cdot))$. 

\begin{definition}[Reidemeister pairing]
Let $G \le \Mod(\Si)$ be a finite subgroup. The {\em Reidemeister pairing} relative to $G$ is the form
\[
\pair{\cdot, \cdot}_G: H_1(\Si;\Z) \times H_1(\Si;\Z) \to \Z[G]
\]
defined by
\[
\pair{x,y}_G = \sum_{g \in G} (x,g\cdot y) g.
\]
If the group $G$ is implicit, we will write simply $\pair{\cdot,\cdot}$. 
\end{definition}

\begin{lemma}
The Reidemeister pairing satisfies the following properties.
\begin{enumerate}
\item $\pair{\cdot, \cdot}$ is $\Z[G]$-linear in the first argument,
\item $\pair{\cdot, \cdot}$ is {\em skew-Hermitian}: $\pair{y,x} = - \overline{\pair{x,y}}$, where $\overline{\cdot}: \Z[G] \to \Z[G]$ is the involution induced by the map $g \mapsto g^{-1}$ on $G$. 
\item The restriction of $\pair{\cdot,\cdot}$ to each isotypic factor of the decomposition (\ref{equation:CW}) is non-degenerate.
\end{enumerate}
\end{lemma}

\begin{lemma}\label{lemma:rhotarget}
Let $\phi \in \Mod(\Si)^G$ be given. Then $\Psi(\phi) \in \Aut(H_1(\Si;\Z), (\cdot, \cdot))$ preserves the Reidemeister pairing $\pair{\cdot, \cdot}_G$. Moreover, $\phi$ preserves each isotypic factor. Thus, $\Psi(\phi)$ belongs to the subgroup
\[
\prod_{U_i} \Aut_G(U_i^{m_i}, \pair{\cdot, \cdot}_G)\le \Aut(H_1(\Si;\Q), (\cdot, \cdot)).
\]
\end{lemma}

The arguments in Section \ref{section:AKmonodromy1} and \ref{section:AKmonodromy2} make use of some of the explicit structure of the Reidemeister pairing $\pair{\cdot, \cdot}_H$ on $H_1(W)$. We record these here for later use.
\begin{lemma}\label{lemma:reidemeisterfacts}
Let $\mathcal S = \{E_1, F_1, \dots, E_{g_0}, F_{g_0}\} \le H_1(W)$.
\begin{enumerate}
\item Any $v \in \mathcal S$ is {\em isotropic}: $\pair{v,v}_H = 0$.
\item Any $v \in \{E_i, F_i\}, w \in \{E_j, F_j\}$ for $i \ne j$ distinct are orthogonal: $\pair{v,w}_H = 0$. 
\item $\pair{E_i, F_i} = 1$ for any $1 \le i \le g_0$.
\end{enumerate}
\end{lemma}
\begin{proof}
These are all direct consequences of the definition of $\pair{\cdot, \cdot}_H$ and the results of Lemma \ref{lemma:intformW}.
\end{proof}

\section{The Atiyah--Kodaira monodromy (I)}\label{section:AKmonodromy1}

\para{Point-pushing diffeomorphisms} In this section we begin our study of the monodromy map $\mu_{X,m}$. This will be formulated in the language of {\em point-pushing diffeomorphisms}.
\begin{definition}
Let $\Sigma$ be a closed surface with marked point $p$, and let $\gamma \in \pi_1(\Sigma,p)$ be a based loop. There is an isotopy $\Pi_t(\gamma)$ of $\Sigma$ that ``pushes'' $p$ along the path $\gamma$ at unit speed. The {\em point push map} $P(\gamma) \in \Mod(\Sigma, p)$ is defined by
\[
P(\gamma) = \Pi_1(\gamma).
\]
\end{definition}

Suppose now that $\gamma$ determines a simple closed curve on $\Sigma$. Let $\gamma_L, \gamma_R$ be the left-hand, resp. right-hand sides of $\gamma$, viewed as simple closed curves on $\Sigma \setminus \{p\}$. 
\begin{fact}\label{fact:ptpush}
For $\gamma$ a simple closed curve,
\[
P(\gamma) = T_{\gamma_L} T_{\gamma_R}^{-1}.
\]
\end{fact}
\begin{proof}
See \cite[Fact 4.7]{FM}. 
\end{proof}

From a topological point of view, the monodromy $\mu$ is closely related to point-pushing maps. Recall from Proposition \ref{proposition:normalized} our construction of the $W$-bundle $E(X,m)\ra B'$. The branched covering $W \to X$ factors as
\[
s \circ t: W \to V \to X,
\]
with $s: V \to X$ an unramified $(\Z/m\Z)^2$-covering and $t: W \to V$ a ramified $\Z/m\Z$-covering branched over $s^{-1}(p)$ for some point $p \in X$. Following the ``global'' treatment of $E^{nn}(X,m)$ given in Section \ref{section:AKconstruction}, we see that $E(X,m)$ arises as a $\Z/m\Z$ {\em fiberwise} branched covering 
\[
E(X,m) \to B' \times V
\]
of the product bundle $B' \times V \to B'$. The branch locus $\Delta \subset B' \times V$ can be described as follows. Let $\Delta_X \subset X \times X$ denote the diagonal. There is a natural covering map
\[
Q: B' \times V \to X \times X
\]
(as both $B'$ and $V$ arise as unbranched covers of $X$), and $\Delta = Q^{-1}(\Delta_X)$.

This description makes the connection with point-pushing maps apparent. Indeed, one can view the bundle $B' \times V \to B'$ as a trivial $V$-bundle equipped with $m^2$ disjoint sections corresponding to $\Delta$. There is then a monodromy representation 
\[
\mu': \pi_1(B') \to \Mod(V,m^2).
\]
The monodromy about some loop $\gamma \in B'$ can be described in terms of point-pushing maps. Under the covering map $B' \to X$, the loop $\gamma \subset B'$ determines a loop on $X$. Taking the preimage of $\gamma$ under $s: V \to X$, one obtains $m^2$ parameterized loops $\gamma_{i,j}(t)$ on $V$, such that for each fixed $t$, the $m^2$ points $\{\gamma_{i,j}(t)\}$ are distinct. The monodromy $\mu'(\gamma)$ is thus a {\em simultaneous multipush} along the curves $\gamma_{i,j}$. More generally, we can apply this construction to any loop $\gamma \subset X$, not merely those $\gamma$ that lift to $B'$.

Given any covering $\Si \to \Si'$ of surfaces, basic topology implies that there is a finite-index subgroup $\LMod \le \Mod(\Si')$ that {\em lifts} to $\Si$, in the sense that there is a homomorphism $\ell: \LMod \to \Mod(\Si)$. The following lemma is immediate from the global topological construction of $E(X,m)$ given above.
\begin{lemma}
Let $\mu': \pi_1(B') \to \Mod(V, m^2)$ be the simultaneous multipush map described above, and let $\LMod \le \Mod(V, m^2)$ be the subgroup admitting a lift $\ell: \LMod \to \Mod(W,m^2)$. Then $ \mu' \mid_{\pi_1(B')} \le \LMod$. Consequently, there is a factorization
\[
\mu = \ell \circ \mu'.
\]
\end{lemma}

\begin{remark}
Let $f \in \Mod(V,m^2)$ lift to $\tilde f \in \Mod(W,m^2)$. Observe that $\zeta^k \tilde f \in \Mod(W,m^2)$ is also a lift for any $k \in \Z$. Thus it is ambiguous to speak of ``the'' lift of an element of $\Mod(V,m^2)$. In the remainder of the section, we will determine explicit formulas for $\mu(\gamma)$ on certain special elements $\gamma \in \pi_1(B')$. To avoid cumbersome notation and exposition, we will ignore this ambiguity wherever possible. In later stages of the argument, it will be necessary to more precisely analyze the effect of this ambiguity; luckily we will see that it is essentially a non-issue.
\end{remark}

%
%

\para{Lifting Dehn twists} The preceding analysis gives a satisfactory description of $\mu'(\gamma)$. To study $\mu$, it therefore remains to understand the lifting map $\ell$. We will be especially interested in a study of lifting Dehn twists $T_c$. Our treatment here follows \cite[Section 3.1]{looijenga}, but there are some crucial differences arising from the fact that we are studying {\em branched} coverings. Recalling that a branched covering of compact Riemann surfaces becomes unbranched after deleting the branch locus, we formulate the results here for unbranched coverings of possibly noncompact Riemann surfaces. 

The following lemma appears in \cite[Section 3.1]{looijenga}. We will analyze when the reverse implication fails to hold in the next subsection. 
\begin{lemma}\label{lemma:twistlift}
Let $q: \Sigma \to \Sigma'$ be an unbranched cyclic $m$-fold covering, classified by $f \in H^1(\Sigma'; \Z/m\Z)$. Let $c$ be a simple closed curve on $\Sigma'$. Then the Dehn twist power $T_c^d$ lifts to $\Mod(\Sigma)$ if (but not necessarily only if) the equation $d\cdot f(c) = 0$ holds in $\Z/m\Z$. 

The preimage $q^{-1}(c)$ has $[\Z/m\Z:\pair{f(c)}]$ components, in correspondence with the set of cosets $(\Z/m\Z)/\pair{f(c)}$. Let $\tilde c$ denote one such component. For $d = \abs{f(c)}$, where $\abs{f(c)}$ denotes the order in $\Z/m\Z$, there is a distinguished lift 
\[
\widetilde{T_c^d} = \prod_{g \in (\Z/m\Z)/\pair{f(c)}} T_{g\cdot \tilde c}.
\]
\end{lemma}

We wish to describe $\Psi(\widetilde{T_c^d}) \in \Sp(H_1(\Sigma; \Z))$. The formula is best expressed using the Reidemeister pairing $\pair{\cdot,\cdot}$ described in Section \ref{section:monodromy}.

\begin{proposition}[Cf. {\cite[(3.1)]{looijenga}}]\label{proposition:twistformula}
With all notation as above, $\Psi(\widetilde{T_c^d}) \in \Sp(H_1(\Sigma;\Z))$ is given by
\begin{equation}\label{equation:twistformula}
\Psi(\widetilde{T_c^d})(x) = x + d^{-1} \pair{x,\tilde c}[\tilde c].
\end{equation}
\end{proposition}

\para{Lifting separating Dehn twists} We return to the setting of Lemma \ref{lemma:twistlift}. Our objective is to understand when a Dehn twist power $T_c^d$ lifts to a mapping class on $\Sigma$ even when the equation $d\cdot f(c) = 0$ fails to hold. This phenomenon is a consequence of the degeneracy of the intersection pairing on a noncompact Riemann surface.

\begin{lemma}\label{lemma:twistliftsep}
Let $\Sigma'$ be a Riemann surface with two or more punctures, and let $q: \Sigma \to \Sigma'$ be an unbranched cyclic covering classified by $f \in H^1(\Sigma'; \Z/m\Z)$. Suppose $c$ is a separating simple closed curve on $\Sigma'$. Then $T_c$ lifts to a diffeomorphism $\widetilde{T_c}$ of $\Sigma$ regardless of the value of $f(c)$. 
\end{lemma}
\begin{proof}
Choose $p \in \Sigma'$ such that $T_c(p) = p$. By covering space theory, $T_c$ lifts to $\Sigma$ if and only if $T_c$ preserves the subgroup $\pi_1(\Sigma) = \ker(f)$ of $\pi_1(\Sigma', p)$. Since the cover $\Si\ra\Si'$ is abelian, it suffices to show that the action of $T_c$ on $H_1(\Sigma', \Z)$ is trivial. The result now follows, since the action of a Dehn twist $T_c$ on homology is given by the transvection formula
\[
x \mapsto x + (x,c)[c],
\]
and $(x,c) = 0$ for all $x$ since $c$ is separating. 
\end{proof}

The diffeomorphism $\widetilde{T_c}$ is ambiguously defined: there are $m$ distinct lifts of $T_c$ to $\Sigma$, each differing by an element of the covering group for $q: \Sigma \to \Sigma'$. To fix a choice, observe that since $c$ is separating, there is a decomposition $\Sigma' = \Sigma_L' \cup \Sigma_R'$ with $\partial(\Sigma_L') = \partial(\Sigma_R') = c$. Choosing an orientation of $c$ distinguishes $\Sigma_L'$ by the condition that $\Sigma_L'$ lie to the left of $c$. There is exactly one lift of $T_c$ to $\Sigma$ such that $q^{-1}(\Sigma_L')$ is pointwise fixed: we take this as our definition of the distinguished lift $\widetilde{T_c}$.\\

The aim of this subsection is again to determine $\Psi(\widetilde{T_c})$. There are various possibilities, depending on the value of $f(c)$. This value depends on a choice of orientation on $c$, which we fix once and for all. Our primary case of interest is when $f(c) = 1 \in \Z/m\Z$. 

\begin{lemma}\label{lemma:twistformulasep1}
Let $c \subset \Sigma'$ be separating, and suppose that $f(c) = 1$. Necessarily $q^{-1}(c)$ is a single separating curve on $\Sigma$ that gives a decomposition $\Sigma = \Sigma_L \cup \Sigma_R$. Then 
\[
H_1(\Sigma; \Z) \simeq H_1(\Sigma_L;\Z) \oplus H_1(\Sigma_R; \Z).
\]
Relative to this decomposition, $\Psi(\widetilde{T_c})$ acts on $H_1(\Sigma_L; \Z)$ trivially, and on $H_1(\Sigma_R;\Z)$ by $\zeta$, where $\zeta$ is the generator of the covering group of $q: \Sigma \to \Sigma'$.
\end{lemma}

\begin{proof}
The action of $\Phi(\wtil T_c)$ on $H_1(\Sigma_L; \Z)$ is trivial because $\widetilde{T_c}$ fixes $\Sigma_L$ pointwise.  Any curve $a \subset \Sigma_R'$ is fixed by $T_c$. Then if $\tilde a \subset \Sigma_R$ is a lift of $a$, then $\widetilde{T_c}(\tilde a) = \zeta^k \tilde a$ for some $k$. To determine $k$, fix a small arc $\alpha$ crossing $c$ from a point $p_L \in \Sigma_L'$ to $p_R \in \Sigma_R'$. Let $\tilde \alpha$ be a lift connecting points $\widetilde{p_L} \in \Sigma_L$ and $\widetilde{p_R} \in \Sigma_R$. By the assumption $f(c) = 1$ and basic covering space theory, the arc $T_c(\alpha)$ lifts to an arc $\widetilde{T_c(\alpha)}$ that connects $\widetilde{p_L}$ to $\zeta \cdot \widetilde{p_R}$. It follows that $\widetilde{T_c}$ acts on $H_1(\Sigma_R; \Z)$ by $\zeta$ as claimed. 
\end{proof}

\para{Lifting a point-push} We now fix our attention on the branched covering $t: W \to V$. Let $b \in L \subset W$ be a branch point, and let $\gamma$ be a simple closed loop on $V$ based at $b$ that is disjoint from $L \setminus \{b\}$. As above, $\gamma$ determines curves $\gamma_L, \gamma_R \subset V^\circ$. We seek a formula for $(\Psi \circ \ell)(P(\gamma^k))$ for $k$ such that $P(\gamma^k)$ lifts to $W^\circ$. 

Since $\gamma_L \cup \gamma_R$ bounds an annulus on $V$ containing $b$, there is an equality in $H_1(V^\circ; \Z/m\Z)$ of the form
\[
[\gamma_R] = [\gamma_L] + C,	
\]
where, as above, $C$ denotes the homology class of a small loop encircling $b$ counterclockwise. Recall that the unbranched covering $t: W^\circ \to V^\circ$ is classified by $\theta \in H^1(V^\circ; \Z/m\Z)$, where $\theta(C) = 1$. It now follows from the discussion of the preceding section that $P(\gamma^k) = T_{\gamma_L}^k T_{\gamma_R}^{-k}$ lifts to $Z^\circ$ if $m \mid k$. 

\begin{lemma}\label{lemma:pushformula} Suppose that $\theta(\gamma_L) = 0$. Then $(\Psi \circ \ell)(P(\gamma^m))$ is given by 
\[
(\Psi \circ \ell)(P(\gamma^m))(x) = x + (m- (1 + \zeta + \dots + \zeta^{m-1})) \pair{x,\widetilde \gamma_L}_t[\widetilde \gamma_L].
\]
Here, $\pair{\cdot, \cdot}_t$ denotes the Reidemeister pairing with respect to the $\pair{\zeta}$-covering $t: W \to V$. 
\end{lemma}
\begin{proof}
A formula for $(\Psi \circ \ell)(P(\gamma^m))$ can be found by applying the results of the previous section. Suppose first that $\theta(\gamma_L) = 0$. Then $\theta(\gamma_R) = 1$. Thus $t^{-1}(\gamma_L)$ consists of $m$ disjoint components, while $t^{-1}(\gamma_R)$ is a single curve. Moreover, the annulus bounded by $\gamma_L, \gamma_R$ on $V$ lifts to a surface with these $m+1$ boundary components. On the level of homology, this implies
\[
(1 + \zeta + \dots + \zeta^{m-1})[\widetilde \gamma_L] = [\widetilde \gamma_R].
\]
From Proposition \ref{proposition:twistformula}, 
\[
\Psi(\widetilde{T_{\gamma_L}^m})(x) = x + m\pair{x,\widetilde \gamma_L}_t[\widetilde \gamma_L],
\]
while
\[
\Psi(\widetilde{T_{\gamma_R}^{-m}})(x) = x - m^{-1} \pair{x, \widetilde\gamma_R}_t[\widetilde \gamma_R].
\]
As $[\widetilde \gamma_R]  = (1+ \dots + \zeta^{m-1})[\widetilde \gamma_L]$, the skew-Hermitian property of the Reidemeister pairing and the equation $(1 + \dots + \zeta^{m-1})^2 = m(1 + \dots + \zeta^{m-1})$ implies the formula. 
\end{proof}

\para{Monodromy of clean elements}
Our analysis of $\mu$ hinges on a study of a special class of elements of $\pi_1(X,x_0)$. 
\begin{definition}[Clean element]\label{definition:clean}
An element $\gamma \in \pi_1(X,x_0)$ is {\em clean} if the following conditions are satisfied.
\begin{enumerate}
\item $\gamma$ has a representative as a simple closed loop on $X$ based at $x_0$,
\item $\gamma \in \pi_1(V) \le \pi_1(X)$,
\item $\theta(\tilde \gamma_L) = 0$ for any (hence all) lifts $\tilde \gamma_L$ of the left-hand curve $\gamma_L$ to $V^\circ$. 
\end{enumerate}
\end{definition}

\begin{remark}\label{remark:sccclean}
The assignment $\gamma \mapsto \gamma_L$ assigns an unbased simple closed curve $\gamma_L \subset X^\circ$ to a based loop $\gamma \subset X$. Observe that conditions (2) and (3) above are well-defined on the level of simple closed curves on $X^\circ$. Moreover, if $\gamma_L \subset X^\circ$ is a simple closed curve for which (2) and (3) hold, then for any choice of representative $\gamma$ of $\gamma_L$ as a simple closed loop based at $x_0 \in X$, the element $\gamma$ is clean. In this way, we can extend the notion of cleanliness to simple closed curves on $X^\circ$. 
\end{remark}

 If $\gamma$ is clean, then $s^{-1}(\gamma) \subset V$ consists of $m^2$ disjoint components which are permuted by the covering group $\pair{\sigma, \tau} \cong (\Z/m\Z)^2$ of $s: V \to X$. Choosing a distinguished lift $\tilde \gamma$, the monodromy $\mu'(\gamma)$ then consists of $m^2$ point-push maps about the disjoint curves $\sigma^i \tau^j \tilde \gamma$. 

\begin{lemma}\label{lemma:clean}
Suppose that $\gamma$ is clean. Then $\gamma^m$ lifts to an element of $\Mod(W)$, in the sense that $\mu(\gamma^m)$ is defined. For any $x \in H_1(W)$,
\[
\rho(\gamma^m)(x) = x + (m-(1 + \zeta + \dots + \zeta^{m-1}))\pair{x,\widetilde{\gamma_L}}_{s \circ t}[\widetilde{\gamma_L}].
\]
Here, $\pair{\cdot, \cdot}_{s \circ t}$ denotes the Reidemeister pairing with respect to the $H$-covering $s \circ t: W \to X$. 
\end{lemma}
\begin{proof}
The monodromy $\rho$ factors as $ \rho = \Psi \circ \ell \circ \mu'$. Topologically, $\mu'(\gamma)$ consists of the simultaneous point-push maps about the $m^2$ disjoint curves $\sigma^i \tau^j \tilde \gamma$. By the $(\Z/m\Z)^2$-symmetry, this is given by
\[
\mu'(\gamma) = \prod_{(i,j) \in (\Z/m\Z)^2} (\sigma^i \tau^j)\cdot P(\tilde \gamma). 
\]
The result now follows from applying Lemma \ref{lemma:pushformula}:
\begin{align*}
\rho(\gamma^m)(x) &= \prod_{(i,j) \in (\Z/m\Z)^2} (\sigma^i \tau^j) (\Psi \circ \ell)(P(\tilde \gamma^m))(x)\\
&= x + \sum_{(i,j) \in (\Z/m\Z)^2} (\sigma^i \tau^j) \cdot (m - (1 + \dots +\zeta^{m-1}))\pair{x, \widetilde{\gamma_L}}_t[\widetilde{\gamma_L}]\\
&= x + (m - (1 + \dots +\zeta^{m-1}))\pair{x, \widetilde{\gamma_L}}_{s \circ t}[\widetilde{\gamma_L}].\qedhere
\end{align*}
\end{proof}

There is a sub-class of clean elements which will be of particular importance.
\begin{definition}[$W$-separating]
A clean element $\gamma$ such that each component of $t^{-1}(\widetilde{\gamma_L})$ is a separating curve on $W$ is said to be {\em $W$-separating}. 
\end{definition}

\begin{lemma}\label{lemma:Vsep}
Suppose $\gamma$ is $W$-separating. Then $\gamma$ (and not merely $\gamma^m$) lifts to an element of $\Mod(W)$. Such $\gamma$ induces a decomposition
\[
H_1(W) = H_1(W') \oplus H_1(W'')
\]
with $W'$ the subsurface of $W$ lying to the {\em left} of $t^{-1}(\gamma_R)$. Relative to this, $\gamma$ acts on $x \in H_1(W)$ via the formula
\[
x\mapsto \begin{cases} x & x \in H_1(W')\\
							\zeta^{-1}\cdot x & x \in H_1(W'').
\end{cases}
\]
\end{lemma}
\begin{proof}
As discussed above,
\begin{align}
\mu'(\gamma) 	&= \prod_{(i,j) \in (\Z/m\Z)^2} (\sigma^i \tau^j)\cdot P(\tilde \gamma)\\
			&= \prod_{(i,j) \in (\Z/m\Z)^2} (\sigma^i \tau^j) \cdot T_{\widetilde{\gamma_L}} T_{\widetilde{\gamma_R}}^{-1}.\label{equation:rho'gamma}
\end{align}
The assumption implies that $\widetilde{\gamma_L}$ and $\widetilde{\gamma_R}$ are both separating on $V$ and hence on $V^\circ$. By Lemma \ref{lemma:twistliftsep} (as applied to the unbranched covering $t: W^\circ \to V^\circ$), both $T_{\widetilde{\gamma_L}}$ and $T_{\widetilde{\gamma_R}}^{-1}$ lift to elements of $\Mod(W)$. As $\gamma$ is clean, $\theta(\widetilde{\gamma_L}) = 0$ by assumption. It follows that $\widetilde{\gamma_L}$ lifts to a collection of $m$ curves on $W$, each of which is separating by assumption. Thus the action of the lift of $T_{\widetilde{\gamma_L}}$ on $H_1(W)$ is trivial.

Since $\theta(\widetilde{\gamma_L}) = 0$, it follows that $\theta(\widetilde{\gamma_R}) = 1$. Lemma \ref{lemma:twistformulasep1} can therefore be applied to give a formula for the action of $T_{\widetilde{\gamma_R}}^{-1}$ on $H_1(W)$. The claimed formula now follows by combining this and (\ref{equation:rho'gamma}).
\end{proof}

\section{Generating arithmetic groups by unipotents}\label{section:unipotents}

In this section we establish the general setup that will allow us to prove arithmeticity of the image. The material here recasts and combines some results from \cite{looijenga} and \cite{venkataramana}. 

Let $G$ be a finite group. Fix a quotient ring $\Q[G]\onto A$; we write the Wedderburn decomposition $A\simeq\prod A_j$. Let $\ca R<A$ be the image of $\Z[G]$ in $A$, and let $\ca R_j<A_j$ be the image of $\Z[G]$ in $A_j$. Note that $\ca R<\prod\ca R_j$ is finite index. Let $M\simeq A^d$ be a free $A$-module with a skew-Hermitian form $\lan\cdot,\cdot\ran:M\times M\ra A$ and automorphism group $\mbf G=\Aut_A(M,\lan\cdot,\cdot\ran)$. The decomposition $A=\prod A_j$ induces a decomposition $M=\bigoplus M_j$ and forms $\lan\cdot,\cdot\ran:M_j\times M_j\ra A_j$, and we denote $\mbf G_j=\Aut_{A_j}(M_j, \lan\cdot,\cdot\ran)$. 

Assume that there are isotropic integral vectors $x_1,x_1^*\in M$ with  $\lan x_1,x_1^*\ran=1$ and assume that each spans a free submodule $A\{x_1,x_1^*\}\simeq A^2$. Let $\ca F,\ca F^-$ be the flags 
\begin{equation}\label{eqn:flags}\ca F \>\>:=\>\> A\{x_1\}\sbs A\{x_1\}^\perp\sbs M\>\>\>\text{ and }\>\>\>\ca F^-\>\>:=\>\>A\{x_1^*\}\sbs A\{x_1^*\}^\perp\sbs M.\end{equation}
and let $\ca P,\ca P^-$ and $\ca U,\ca U^-$ be the corresponding parabolic and unipotent subgroups of $\mbf G$. Specifically, $\ca P$ is the group that preserves the flag $\ca F$, and $\ca U<\ca P$ is the subgroup that acts trivial on successive quotients $M/A\{x_1\}^\perp$ and $A\{x_1\}^\perp/A\{x_1\}$. The groups $\ca P^-$ and $\ca U^-$ are defined similarly. After choosing a basis $(x,\ld,x^*)$, we can write
\begin{equation}\label{eqn:unipotent}\ca U=\{g=\left(\begin{array}{ccccc}
1&-(Q\bar v)^t&w\\
0&I&v\\
0&0&1
\end{array}\right): v\in A^{d-2}\text{ and }\lan v,v\ran=\bar w-w
\},\end{equation}
where $Q$ is the matrix of $\lan\cdot,\cdot\ran$ restricted to $A\{x_1,x_1^*\}^\perp\simeq A^{d-2}$ with respect to the given basis. Denoting $\ov{\ca U}:=A^{d-2}$, there is a surjection
\begin{equation}\label{eqn:project-unipotent}
\begin{array}{rcl}
\pi:\ca U&\ra& \ov{\ca U}\\
g&\mapsto& v.
\end{array}\end{equation}

\begin{proposition}[Generated by enough unipotents]\label{proposition:UU}
We use the notation of the preceding paragraphs. Fix $\Ga<\mbf G(\ca R)\doteq\prod \mbf G_j(\ca R_j)$. Assume that $A\{x_1,x_1^*\}^\perp$ contains isotropic vectors $x_2,x_2^*$ with $\lan x_2,x_2^*\ran=1$ that each span a free $A$-submodule. Suppose that the images of $\Ga\cap\ca U(\ca R)\ra \ov{\ca U}(\ca R)$ and $\Ga\cap\ca U^-(\ca R)\ra \ov{\ca U}{}^-(\ca R)$ have finite index in $\ov{\ca U}(\ca R)$ and $\ov{\ca U}{}^-(\ca R)$ respectively (as abelian groups). Then 
\begin{enumerate}[(i)]
\item the image of $\Ga\ra \mbf G_j(\ca R_j)$ is finite index in $\mbf G_j(\ca R_j)$ for each $j$, and 
\item $\Ga$ has finite index in $\mbf G(\ca R)$. 
\end{enumerate} 
\end{proposition}

Of course (ii) implies (i), but in order to prove (ii) we will use (i). The proof of (i) will follow quickly from the following Theorem \ref{thm:venky} of \cite[Cor.\ 1]{venkataramana}, which builds off work of Tits, Vaserstein, Raghunathan, Venkataramana, and Margulis. We will also need the following lemma. 

\begin{lemma}[Finite-index subgroups of $\ca U(\ca R)$]\label{lem:generate-unipotent}
Let $\ca U$ and $\pi:\ca U\ra\ov{\ca U}$ be as in (\ref{eqn:unipotent}) and (\ref{eqn:project-unipotent}). Assume there are $y,y^*\in \ca R^{d-2}\sbs A\{x_1,x_1^*\}^\perp$ with $\lan y,y^*\ran=1$. Then for $\Lam<\ca U(\ca R)$, if $\pi(\Lam)$ is finite index in $\ov{\ca U}(\ca R)\simeq \ca R^{d-2}$ (as an abelian group), then $\Lam$ is finite index in $\ca U(\ca R)$. 
\end{lemma}
\begin{proof}
There is an exact sequence $0\ra \ca R_0\ra\ca U(\ca R)\ra\ov{\ca U}(\ca R)\ra0$, where $\ca R_0=\{w\in \ca R: \bar w=w\}$. A subgroup $\Lam<\ca U(\ca R)$ is finite index if and only if $\Lam\cap \ca R_0$ is finite index is $\ca R_0$ and $\pi(\Lam)$ is finite index in $\ov{\ca U}(\ca R)$. Since we're assuming the latter, we need only show the former. 

We can identify $\ca U(\ca R)\simeq \ca R^{d-2}\ti \ca R_0$ (as sets) via  
\[\left(
\begin{array}{ccc}
1&-(Q\bar v)^t&z-\frac{1}{2}\lan v,v\ran\\
0&I&v\\
0&0&1
\end{array}
\right)\leftrightarrow (v,z).\]
Under this bijection, the multiplication on $\ca U(\ca R)$ becomes 
\[(u,z)\cdot(u',z')=\big(u+u', z+z'+\de(u,u')\big),\]
where $\de(u,u')=\frac{1}{2}\big[\lan u,u'\ran -\lan u',u\ran \big]$. 

With these coordinates, $(u,z)^{-1}=(-u,-z)$, and the commutator of $(u,z)$ and $(u',z')$ is 
\[\big[(u,z),(u',z')\big]=\big(0,2\de(u,u')\big)\]
By assumption, there exists $y,y^*\in \ca R^{d-2}$ with $\lan y,y^*\ran=1$. Let $\{\al_i\}$ be a finite generating set of $\ca R_0$ as an abelian group. Since $\pi(\Lam)<\ov{\ca U}(\ca R)$ is finite index, there exists $\ell>0$ and $z_i,z\in\ca R_0$ so that $h=(\ell y^*,z)\in \Lam$ and $g_i=(\ell\al_iy,z_i)\in \Lam$ for every $i$. Then 
\[[g_i,h]=\big(0,2\de(\ell\al_iy,\ell y^*)\big)=(0,2\ell^2\al_i).\]
In particular, $\Lam\cap\ca R_0$ contains the subgroup generated by $\{2\ell^2\al_i\}$, which is finite index in $\ca R_0$. 
\end{proof}

\begin{theorem}[Corollary 1 in \cite{venkataramana}]\label{thm:venky}
Suppose $\mbf G$ is an algebraic group over $K$ that is absolutely simple and has $K$-rank $\ge2$. Let $\ca P$ and $\ca P^-$ be opposite parabolic $K$-subgroups, and let $\ca U,\ca U^-$ be their unipotent radicals. Denoting $O_K\sbs K$ the ring of integers, for any $N\ge1$, the group $\De_N(\ca P^\pm)$ generated by $N$-th powers of $\ca U(O_K)$ and $\ca U^-( O_K)$ is finite index in $\mbf G(O_K)$. 
\end{theorem}

\begin{proof}[Proof of Proposition \ref{proposition:UU}]
To prove (i) we apply Theorem \ref{thm:venky} to $\mbf G_j$ for each $j$. 

First we identify $\mbf G_j$ as an algebraic group. To do this, it will help to first recall the structure of $A_j$. According to Wedderburn's theorem, each $A_j$ is isomorphic to a matrix algebra $M_{k}(\De)$ over a division ring $\De$ (of course $n$ and $\De$ depend on $j$, but we omit $j$ from the notation). The center $L:=Z(\De)$ is a number field, and we'll denote $K<L$ the subfield fixed by the involution (either $K=L$ or $[L:K]=2$).

The group $\mbf G_j=\Aut_{A_j}(M_j,\lan\cdot,\cdot\ran)$ can be identified with matrices $g\in M_d(A_j^{\opp})\simeq\End_{A_j}(M_j)$ with $g^tQ_j\bar g=Q_j$, where $Q_j\in M_r(A_j)$ is the matrix for $\lan\cdot,\cdot\ran:M_j\ti M_j\ra A_j$ with respect to a given basis. Given the isomorphism $M_d(A_j)\simeq M_{dk}(\De)$, we can also view $\mbf G_j$ as the automorphism group of a non-degenerate skew-Hermitian form on $\De^{dk}$. There is a homomorphism $\mbf G_j\sbs M_{dk}(\De)\hra M_{dkr}(K)$ induced from a linear map $\De\ra M_r(K)$ defined by left multiplication of $\De$ on $\De\simeq K^r$. Given $\mbf G_j\hra M_{dkr}(K)$ it is easy to deduce that $\mbf G_j$ is an algebraic group over $K$. In fact, $\mbf G_j$ is one of the classical groups and is an absolutely almost simple over $K$ (see \cite[\S2.3.3]{platonov-rapinchuk} and \cite[\S18.5]{witte-morris} for more details). Furthermore, $\mbf G_j(O_K)$ is commensurable with $\mbf G_j(R_j)$, and the $K$-rank of $\mbf G_j$ is at least 2, since by our assumption $M$ contains a 2-dimensional isotropic subspace $A\{x_1,x_2\}\simeq A^2$. 

The subgroups $\ca P,\ca P^-<\prod \mbf G_j$ project to opposite parabolic subgroups $\ca P_j,\ca P_j^-<\mbf G_j$, and $\ca U,\ca U^-$ project to the corresponding unipotent radicals $\ca U_j<\ca P_j$. Since $\ca U(\ca R)\doteq\prod\ca U_j(\ca R_j)$ and $\ov{\ca U}(\ca R)\doteq\prod\ov{\ca U}_j(\ca R_j)$ are commensurable, by our assumption, the image of $\Ga\cap\ca U_j(\ca R_j)\ra\ov{\ca U}_j(\ca R_j)$ is finite index in $\ov{\ca U}_j(\ca R_j)$ for each $j$. By Lemma \ref{lem:generate-unipotent}, $\Ga\cap\ca U_j(\ca R_j)$ is finite index in $\ca U_j(\ca R_j)$. Similarly, $\Ga\cap\ca U_j{}^-(\ca R_j)$ is finite index in $\ca U_j{}^-(\ca R_j)$, and so by Theorem \ref{thm:venky} the image of $\Ga$ in $\mbf G_j(\ca R_j)$ is finite index.

Now we address (ii). To show that $\Ga$ has finite index in $\mbf G(\ca R)\doteq\prod \mbf G_j(\ca R_j)$, we will show $\Ga$ contains $\prod \Lam_j$, where $\Lam_j<\mbf G_j(\ca R_j)$ is finite index for each $j$. Let $\Ga_j$ be the image of $\Ga$ in $\mbf G_j$. By (i), we know $\Ga_j$ is a lattice. Let $\hat\Ga_j$ be the kernel of $\Ga\cap \mbf G\ra \prod_{i\neq j}\mbf G_i$. Observe that $\hat\Ga_j<\Ga_j$ is a normal subgroup. By the Margulis normal subgroups theorem, $\hat\Ga_j$ is either finite or finite index in $\Ga_j$. Thus to prove (ii) it suffices to show that $\hat\Ga_j$ contains an infinite order element for each $j$. 

By assumption, we have isotropic vectors $x_2,x_2^*\in M'\simeq \ca R^{d-2}$ with $\lan x_2,x_2^*\ran=1$. For simplicity denote $y=x_2$ and $y^*=x_2^*$, and let $y_j,y_j^*$ denote the projection to $M_j'\simeq \ca R_j^{d-2}$. Note that $\lan y_j,y_j^*\ran$ is equal to the identity element $e_j\in\ca R_j$. 

Since the image of $\Ga\cap\ca U_j(\ca R_j)\ra\ov{\ca U}_j(\ca R_j)\simeq\ca R_j^{d-2}$ is finite index, there exists $\ell>0$ and $z,z^*\in\ca R_0$ so that $(\ell y_j,z)$ and $(\ell y_j^*,z^*)$ belong to $\Ga\cap\ca U(\ca R)$. By the computation from Lemma \ref{lem:generate-unipotent}, the commutator of $(\ell y_j,z)$ and $(\ell y_j^*,z^*)$ is $(0,2\ell^2e_j)$. This element of $\Ga$ has infinite order and is in the kernel of $G\ra\prod_{i\neq j}\mbf G_i$. This completes the proof. 
\end{proof}

We end this section with a few lemmas about the algebraic structure of $\ca P$ and $\ca U$. These results will be essential for our computation in Section \ref{section:AKmonodromy2}. Set 
\[M'=A\{x_1,x_1^*\}^\perp.\] Note that $A\{x_1\}^\perp=A\{x_1\}\op M'$. 

\begin{lemma}[Commutator trick]\label{lemma:commtrick}
Fix $v=ax_1+u\in A\{x_1\}\op M'$ with $u$ isotropic, and define $T_v:M\ra M$ by $T_v(x)=x+\lan x,v\ran v$. Fix nonzero, central $\ze\in A$, and define $R:M\ra M$ by 
\[R:x\mapsto\left\{\begin{array}{cl}x&x\in A\{x_1,x_1^*\}\\[1mm]\ze^{-1} x&x\in M'.\end{array}\right.\]
Then $T_v,R\in\ca P$ and $[T_v,R]\in\ca U$. Furthermore, $\pi\big([T_v,R]\big)=\bar a(\ze^{-1}-1)u$. 
\end{lemma}

\begin{lemma}[Parabolic action on unipotent]\label{lemma:parabolic}
Fix  
\[
g=\left(\begin{array}{ccc}
1&-(Q\bar u)^t&w\\
0&I&u\\
0&0&1
\end{array}\right)\in\ca U
\>\>\>\text{ and }\>\>\>h=\left(\begin{array}{ccc}
a&*&*\\
0&B&*\\
0&0&c
\end{array}\right)\in\ca P.\] Then $hgh^{-1}\in\ca U$ and $\pi\big(hgh^{-1}\big)=Buc^{-1}$. 
\end{lemma}

Both lemmas follow from direct computation. For Lemma \ref{lemma:parabolic}, it's useful to recall the Levi decomposition $\ca P=\ca M\>\ca U$, where $\ca M$ consists of block diagonal matrices.

\section{Representations of finite Heisenberg groups}\label{section:heisenberg}

Fix $m\ge2$, and let $H=\scr H(\Z/m\Z)$ be the Heisenberg group, c.f.\ (\ref{equation:Hpres}). Here we detail the representation and character theory of $H$ over $\C$ and $\Q$. Our main interest in this is to obtain information about the decomposition $H_1(W;\Q)=\bigoplus_{k\mid m}\bigoplus_{\chi}M_{k,\chi}$ into isotypic factors. This will be needed to prove Theorem \ref{theorem:main-nn}. 

\begin{proposition}[Representations of $H$]\label{proposition:heisenberg-reps}
Let $\phi$ be the Euler totient function. 
\begin{enumerate}
\item[(a)] Fix $k\mid m$. There are $(m/k)^2\cdot\phi(k)$ simple $\C[H]$-modules of dimension $k$ (up to isomorphism). They are indexed $U_{a,b,c}$ for $a,b\in\Z/(\frac{m}{k})\Z$ and $c\in(\Z/k\Z)^\times$. Furthermore, varying over $k$, these account for all the simple $\C[H]$-modules.
\item[(b)] Fix $U=U_{a,b,c}$ of dimension $k$, and let $\chi$ be its character. The trace field $\Q(\chi)$ is isomorphic to the cyclotomic field $\Q(\ze_L)$, where $L=\lcm\left(k,\frac{m/k}{\gcd(a,b,m/k)}\right)$. The sum over the orbit of $\chi$ under the Galois group $\Gal(\Q(\ze_L)/\Q)$ is the character of an irreducible $H$-representation over $\Q$. 
\item[(c)] Let $I_k$ be the set of characters of irreducible $k$-dimensional $H$ representations over $\C$, and let $\bar I_k$ be the quotient of $I_k$ by the action of $\Gal(\Q(\ze_m)/\Q)$. Then $\Q[H]$ decomposes into simple algebras
\[\Q[H]\simeq\prod_{k\mid m}\prod_{[\chi]\in\bar I_k} M_k(\Q(\chi))\] (here $M_k(R)$ denotes the algebra of $k \times k$ matrices over the ring $R$).

\end{enumerate} 
\end{proposition}

The decomposition of the group ring into simple algebras is a particular instance of the \emph{Wedderburn decomposition} of a semisimple algebra. 


\begin{proof}[Proof of Proposition \ref{proposition:heisenberg-reps}]
To begin we describe the simple $\C[H]$-modules of dimension $k$, or equivalently the $k$-dimensional irreducible $H$-representations over $\C$. Set $\ell=m/k$, and fix $a,b\in\Z/\ell\Z$ and $c\in (\Z/k\Z)^\times$. Define a representation $\rho=\rho_{a,b,c}:H\ra\GL_k(\C)$ by 
\begin{equation}\label{equation:heisenberg-rep}
\rho(\si)=
\ze_m^a\left(
\begin{array}{ccccc}
0&0&\cdots&0&1\\
1&0&\cdots&0&0\\
0&1&\cdots&0&0\\
\vdots&&\ddots&&\vdots\\
0&0&\cdots&1&0
\end{array}
\right)
\hspace{.2in} 
\rho(\tau)=\ze_m^b
\left(\begin{array}{ccccccccc}
1&&\\
&\ze_k^c&&\\
&&\ddots&\\
&&&\ze_k^{(k-1)c}
\end{array}\right)
\hspace{.2in} 
\rho(\ze)=\ze_k^{-c}\cdot\id
\end{equation}
It is easy to see that these representations are irreducible for each choice of $a,b,c$, and by looking at their characters one finds that no two are isomorphic. Since 
\[\sum_{k\mid m}\sum_{\substack{a,b\in\Z/\frac{m}{k}\Z\\c\in(\Z/k\Z)^\times}}\dim(\rho_{a,b,c})^2=\sum_k \phi(k)\cdot(m/k)^2\cdot k^2=m^3=\dim \C[H],\] we conclude that these are all the simple $\C[H]$-modules. This proves (a). 

We must next determine the trace field $\Q(\chi)$. Note that $\chi(\si^p\tau^q\ze^r)=0$  unless $k$ divides both $p$ and $q$. Furthermore, $\rho(\si^{kp'}\tau^{kq'}\ze^r)=\alpha\cdot\id$ with $\al=\ze_\ell^{ap'+bq'}\ze_k^{-cr}$. It follows that $\Q(\chi)=\Q(\ze_k,\ze_{\ell'})$, where $\ell'=\frac{\ell}{\text{gcd}(a,b,\ell)}$. Since $\Q(\ze_k,\ze_{\ell'})=\Q(\ze_L)$, where $L=\text{lcm}(k,\ell')$, this proves the first part of (b). 

Next we describe the simple $\Q[H]$-modules. References for this are \cite[\S12]{serre} and \cite[\S9-10]{isaacs}. We continue to fix the character $\chi$ of the $k$-dimensional simple $\C[H]$-module $U_{a,b,c}$. The character 
\[\hat\chi=\sum_{\ep\in \Gal(\Q(\chi)/\Q)}\chi^\ep\]
is invariant under $\Gal(\Q(\chi)/\Q)$, and so it is $\Q$-valued. Then $m\cdot\chi$ is the character of a simple $\Q[H]$-module, where $m=m_\Q(\chi)$ is the \emph{Schur index}. According to a theorem of Roquette (see \cite[Cor.\ 10.14]{isaacs} and \cite[Thm.\ 4.7]{jespers}), $m_\Q(\chi)=1$ for every irreducible character of $H$. This is a special fact about nilpotent groups; if $2\mid m$, we also need the fact that $H(\Z/2^i\Z)$ does not admit a split surjection $H\ra Q_8$ to the quaternion group of order 8. 

Now the Wedderburn decomposition for $\Q[H]$ can be determined. $\Q[H]$ decomposes as a product of simple algebras $M_k(\De)$, one for each simple $\Q[H]$-module. Here $\De$ is a division algebra over $\Q(\chi)$ and $\dim_{\Q(\chi)}\De=m^2$, where $m$ is again the Schur index \cite[\S12.2]{serre}. Since the Schur index is always 1, this proves (c). 
\end{proof}

As a consequence of Proposition \ref{proposition:heisenberg-reps}, for any $\Q[H]$-module $W$, the decomposition of $U$ into isotypic factors has the form  $W=\bigoplus_{k\mid m}\bigoplus_{\chi\in\bar I_k}W_{k,\chi}$.


\para{Abelian and nonabelian representations} In studying $\Ga(X,m)$, we are mainly interested in the $H$-representations that are \emph{nonabelian}. We call a representation of $H$ (over $\C$ or $\Q$) \emph{abelian} if it factors through the abelianization $H^{\ab}\simeq(\Z/m\Z)^2$. These are precisely the representations where $\ze$ acts trivially. For example, over $\C$ the irreducible abelian representations are the 1-dimensional representations $U_{a,b,1}$ with $a,b\in\Z/m\Z$. For any $H$-representation $W$, multiplication by $\Pi_{na}=m-(1+\ze+\cdots+\ze^{m-1})\in\Q[H]$ defines a projection $W\ra \Pi_{na}W$ onto the subspace of nonabelian isotypic factors.

If $m$ is prime, then there are $\phi(m)=m-1$ nonabelian irreducible $H$-representations over $\C$. They all have dimension $m$, and $\Pi_{na}\C[H]\simeq M_m(\C)^{\ti m-1}$. Over $\Q$, there is a single nonabelian irreducible representation. It has dimension $m(m-1)$ and $\Pi_{na}\Q[H]\simeq M_m(\Q(\ze_m))$. 

When $m$ is composite, the expression for $\Pi_{na}\Q[H]$ is more complicated. For example, if $m=4$, then $\Pi_{na}\Q[H]\simeq M_2(\Q)^{\ti 4}\times M_4(\Q(i))$, and for $m=6$, 
\[\Pi_{na}\Q[H]\simeq M_2(\Q)\times M_2(\Q(\ze_3))^{\ti 4}\times M_3(\Q(\ze_3))^{\times 4}\times M_6(\Q(\ze_6)).\]

\para{$\tau$-invariants of $H$-representations} In Section \ref{section:nonnormal} we will use the following proposition. Recall the subgroup $Q \le H$ is the cyclic group generated by $\zeta$.

\begin{proposition}[$\tau$-invariants of {$\Q [H]$}-modules]\label{proposition:QH-tau-invariants}
Let $\rho_{a,b,c}:H\ra\GL_k(\C)$ be the irreducible representation (\ref{equation:heisenberg-rep}). Denote $U$ the corresponding vector space. 
\begin{enumerate} 
\item The $\tau$-invariant subspace $U^{\pair\tau}$ is nontrivial if and only if $b=0$. If $b=0$, then $U^{\pair\tau}$ is a 1-dimensional representation of $\C[Q]$ where $\ze$ acts with order $k$. 
\item For $k=m$, there is a single simple $\Q[H]$-module where $\ze$ acts with order $m$. For $1<k<m$, there are non-isomorphic simple $\Q[H]$-modules  $U_1,U_2$ such that $U_j^{\pair{\tau}}\neq0$ and $\ze$ acts on $U_j$ with order $k$. 
\end{enumerate}
\end{proposition} 

\begin{proof}
Claim (1) can be deduced directly from the description of the representation given in (\ref{equation:heisenberg-rep}). 

We prove Claim (2). First note that $I_m$ has $m-1$ elements, which are permuted transitively by $\Gal(\Q(\ze_m)/\Q)\simeq(\Z/m\Z)^\times$. This explains the first sentence of the claim. For the second part, it is not hard to see that if $1<k<m$, then the $k$-dimensional representations $U_{0,0,1}$ and $U_{1,0,1}$ are in different orbits of $\Gal(\Q(\ze_m)/\Q)$. Then these two representations give rise to distinct irreducible $H$-representations over $\Q$ with the desired properties. 
\end{proof}

\section{The Atiyah--Kodaira monodromy (II)}\label{section:AKmonodromy2}

In this section we give a detailed analysis of the monodromy of the Atiyah--Kodaira bundle, culminating in the proof of Theorem \ref{theorem:main}.
\subsection{The image of $\mu$} We return to the setting of Section \ref{section:monodromy}. As established in Lemma \ref{lemma:rhotarget}, the monodromy group $\Gamma=\Ga(X,m)$ is a subgroup of the product
\begin{equation}\label{equation:product}
\prod_{U_i} \Aut_H(U_i^{m_i}, \pair{\cdot, \cdot}_H)\le \Aut(H_1(W;\Q), (\cdot, \cdot)),
\end{equation}
where here $H = \mathscr{H}(\Z/m\Z)$ as usual and the product $U_i$ runs over the isomorphism classes of simple $\Q[H]$-modules $U_i$. (We could also use Section \ref{section:heisenberg} to write the left-hand side of (\ref{equation:product}) as $\prod_{k,\chi}\Aut_H (M_{k,\chi},\pair{\cdot,\cdot}_H)$, but that won't be necessary in this section.) Recall that we say that is $U_i$ {\em abelian} if the $H$-action on $U_i$ factors through the abelianization $H^{\ab} \cong (\Z/m\Z)^2$; otherwise $U_i$ is said to be {\em nonabelian}.

\begin{lemma}\label{lemma:gammaimage}
The projection of $\Gamma$ to any factor of (\ref{equation:product}) corresponding to an abelian $U_i$ is trivial. Consequently,
\[
 \Gamma \le \prod_{U_i \mbox{ nonabelian}} \Aut(U_i^{m_i}, \pair{\cdot, \cdot}_H).
\]
\end{lemma}
\begin{proof}
Let $K \le H$ be any subgroup. Associated to $K$ is the intermediate cover $W \to W_K$ with covering group $K$. Transfer provides an isomorphism 
\[
H_1(W;\Q)^K \cong H_1(W_K; \Q).
\]
Moreover, this isomorphism is compatible with the decomposition
\[
H_1(W;\Q) \cong \bigoplus_{U_i} U_i^{m_i},
\]
so that
\[
H_1(W_K;\Q) \cong \bigoplus_{U_i} (U_i^K)^{m_i}.
\]
Let $K = \pair{\zeta}$; it is easy to see that $K = [H,H]$. In the notation of Section \ref{section:AKconstruction}, the associated surface $W_K$ is given by $V$. As $K$ is also central in $H$ and hence acts by scalars on any simple $\Q[H]$-module $U_i$, it follows that the $K$-invariant space $U_i^K$ is nontrivial if and only if $U_i$ is abelian. This implies that
\[
H_1(V;\Q) \cong \bigoplus_{U_i \mbox{ abelian}} (U_i)^{m_i}.
\]

To summarize, the action of $\Gamma$ on the summand of $H_1(W;\Q)$ corresponding to abelian representations $U_i$ is governed by the monodromy action on the intermediate cover $V$. To prove the claim, it therefore suffices to show that this action is trivial. 

This is easy to see. The cover $V \to X$ is regular with covering group $H/K \cong (\Z/m\Z)^2$. By construction, this is the maximal {\em unramified} cover intermediate to $W \to X$. To study the monodromy action on the fiber $V$, we pass to the punctured surface $V^\circ$. The monodromy action on $V^\circ$ is the lift of simultaneous multi-pushes on $X^\circ$. Since the cover $V \to X$ is unramified, these lift on $V^\circ$ to simultaneous multi-pushes. As is well-known, these diffeomorphisms act trivially on $H_1(V)$, since they become isotopic to the identity after passing to the inclusion $V^\circ \to V$. 
\end{proof}

\subsection{Producing unipotents} Lemma \ref{lemma:gammaimage} identifies an ``upper bound'' for the monodromy group $\Gamma$. Theorem \ref{theorem:main} then asserts that $\Gamma$ is in fact an arithmetic subgroup of this upper bound. The proof of Theorem \ref{theorem:main} will follow from Proposition \ref{proposition:UU}. The first step in the argument is to give an explicit description of the unipotent subgroups $\mathcal U(\mathcal R)$ and $\mathcal U^-(\mathcal R)$ (as well as their abelian quotients $\overline{\mathcal U}(\mathcal R)$ and $\overline{\mathcal U}^-(\mathcal R)$) appearing in the statement of Proposition \ref{proposition:UU}.

We specialize the discussion of Section \ref{section:unipotents} to the situation at hand. Recall that $\Pi_{na}=(m-(1+\ze+\cdots+\ze^{m-1}))\in\Z[H]$. In the notation of Section \ref{section:unipotents}, we take $A = \Pi_{na} \Q[H]$. Then $\mathcal R = \Pi_{na}\Z[H]$. We also take $M = \Pi_{na}H_1(W;\Q)$. Note that $M$ is a {\em free} $A$-module by Lemma \ref{lemma:Mfree}.

\begin{lemma}\label{lemma:ourUU}
In the notation of Section \ref{section:AKconstruction}, consider the elements $E_2, F_2, E_3, F_3 \le H_1(W;\Q)$. Then the following hold:
\begin{enumerate}
\item Each such element is isotropic,
\item $\pair{E_2, F_2}_H= \pair{E_3, F_3}_H=1$,
\item $\{E_3, F_3\} \subset \Q[H]\{E_2,F_2\}^\perp$,
\item $\Q[H]\{E_2,F_2\} \cong \Q[H]\{E_3, F_3\} \cong \Q[H]^2$.
\end{enumerate}
Consequently, the unipotent subgroups $\mathcal U(\mathcal R)$ and $\mathcal U^-(\mathcal R)$ associated to the flags 
\[
\mathcal F = \Q[H]\{E_2\} \subset \Q[H]\{E_2\}^\perp \subset H_1(W;\Q) \quad \mbox{and} \quad \mathcal F^- = \Q[H]\{F_2\} \subset \Q[H]\{F_2\}^\perp \subset H_1(W;\Q)
\]
satisfy the hypotheses of Proposition \ref{proposition:UU}.
\end{lemma}
\begin{proof}
Claims (1)-(3) follow from Lemma \ref{lemma:reidemeisterfacts}, while (4) follows from Lemma \ref{lemma:H1W}.
\end{proof}

The following lemma establishes a direct-sum decomposition for the abelian quotients $\overline{\mathcal U}(\mathcal R)$ and $\overline{\mathcal U}^-(\mathcal R)$. The proof of Theorem \ref{theorem:main} will handle each summand in turn.
\begin{lemma}\label{lemma:ourUbar}
Fix the flags $\mathcal F, \mathcal F^-$ as in Lemma \ref{lemma:ourUU}. Define the following submodules of $\Pi_{na} H_1(W;\Z)$ spanned by the indicated elements.
\begin{align*}
M_1 &= \mathcal R\{E_i, F_i \mid i \ge 3\} \\
M_2 &= \mathcal R\{E_1, F_1\}\\
M_3 &= \mathcal R\{G_h, G_v\}
\end{align*}
Then $M_1 + M_2 + M_3$ is a subgroup of finite index in both $\overline{\mathcal U}(\mathcal R)$ and $\overline{\mathcal U}^-(\mathcal R)$. 
\end{lemma}
\begin{proof}
According to Lemma \ref{lemma:Mfree}, there is an isomorphism of $A$-modules 
\[
\Pi_{na}H_1(W;\Q) \cong A^d.
\]
Consequently, $\Pi_{na} H_1(W;\Z)$ is commensurable to $\mathcal R^d$. Lemma \ref{lemma:H1W} implies that $\Pi_{na} H_1(W;\Z)$ is spanned as an $\mathcal R$-module by the elements $\{E_i, F_i \mid 1 \le i \le g_0\} \cup \{G_h, G_v\}$. By our choice of flag and the definitions of $\overline{\mathcal{U}}(\mathcal R)$ and $\overline{\mathcal{U}}^-(\mathcal R)$, it follows that $\overline{\mathcal{U}} (\mathcal R)$ and $\overline{\mathcal{U}}^-(\mathcal R)$ are spanned as an $\mathcal R$-module by the elements 
\[
\{E_i, F_i \mid 1 \le i \le g_0,\ i \ne 2\} \cup \{G_h, G_v\}.
\]
This generating set is partitioned into the three pieces corresponding to the generators for $M_1, M_2, M_3$. The result follows. 
\end{proof}

\para{Curve-arc sums} In order to apply Proposition \ref{proposition:UU}, it is necessary to produce a large number of unipotent elements. The first step towards this is to produce a large number of parabolic elements; then Lemma \ref{lemma:commtrick} can be used to convert these into unipotents. Following the analysis of Section \ref{section:AKmonodromy1}, we can construct a wide variety of transvections as the image of clean elements $\rho(\gamma^m)$. In order to make the subsequent work with (fairly elaborate) simple closed curves as painless as possible, we introduce here two operations on curves.

The first of these is the {\em curve-arc sum} procedure. Let $\Sigma$ be a surface, and $\gamma_1, \gamma_2$ be disjoint {\em oriented} simple closed curves on $\Sigma$. Let $\alpha$ be an arc on $\Sigma$ beginning at some point on the left side of $\gamma_1$ and ending on the left side of $\gamma_2$ that is otherwise disjoint from $\gamma_1 \cup \gamma_2$. The {\em curve-arc sum} of $\gamma_1$ and $\gamma_2$ along $\alpha$ is the simple closed curve $\gamma_1 +_\alpha \gamma_2$ defined pictorially in Figure \ref{figure:CAS}. 

\begin{figure}[h]
\labellist
\small
\pinlabel $\gamma_1$ [tl] at 56 29.6
\pinlabel $\alpha$ [b] at 83.2 62.4
\pinlabel $\gamma_2$ [l] at 122.4 60
\pinlabel $\gamma_1+_\alpha \gamma_2$ [b] at 275.2 64
\endlabellist
\includegraphics{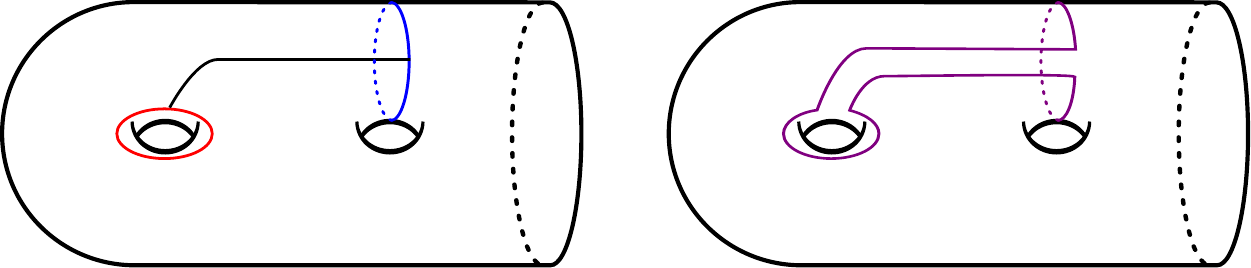}
\caption{The curve-arc sum (the orientations on $\gamma_1, \gamma_2$ are implicit).}
\label{figure:CAS}
\end{figure}
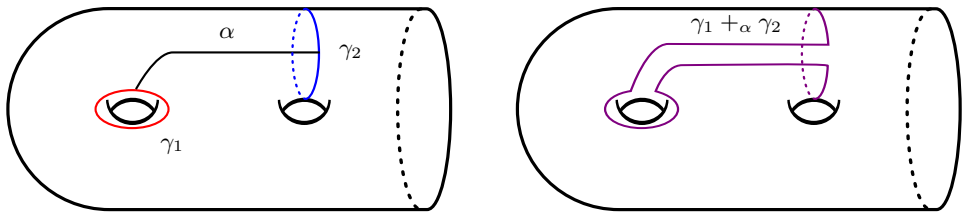

\begin{lemma}\label{lemma:casumfacts}\ 
\begin{enumerate}
\item For oriented simple closed curves $\gamma_1, \gamma_2 \subset \Sigma$ and an arc $\alpha$ connecting $\gamma_1$ and $\gamma_2$, 
\[
[\gamma_1 +_\alpha \gamma_2] = [\gamma_1]+ [\gamma_2]
\]
as elements of $H_1(\Sigma, \Z)$. 
\item Suppose that $\gamma \in \pi_1(X)$ is clean, and that $\delta \subset X^\circ$ is clean in the sense of Remark \ref{remark:sccclean}. Let $\alpha \subset X^\circ$ be any arc such that $\gamma+_\alpha \delta$ is a simple closed curve. Then $\gamma+_\alpha \delta$ is clean.
\end{enumerate}
\end{lemma}
\begin{proof}
(1) is immediate. For (2), it is necessary to check the conditions of Definition \ref{definition:clean}. Condition (1) holds by hypothesis. Condition (2) holds by the fact that $s: V \to X$ is an abelian covering, in combination with Lemma \ref{lemma:casumfacts}.1. This implies that any component of the preimage $s^{-1}(\gamma +_\alpha \delta)$ is itself a curve-arc sum on $V^\circ$. Then Condition (3) follows from the fact that $t: W^\circ \to V^\circ$ is also an abelian covering, again appealing to Lemma \ref{lemma:casumfacts}.1. 
\end{proof}

\para{De-crossing} The second operation we will require is {\em de-crossing}. Suppose that $\gamma \subset \Sigma$ is non-simple, with a self-intersection at $p \in \Sigma$. Suppose that $S\subset \Sigma$ is a subsurface with $S \cong \Sigma_{1,1}$, such that $S \cap \gamma$ contains only the self-intersection at $p$. Then the {\em de-crossing of $\gamma$ along $S$} is the curve $DC(\gamma, S)$ with one fewer self-intersection depicted in Figure \ref{figure:decross}.
\begin{figure}[h]
\labellist
\small
\endlabellist
\includegraphics{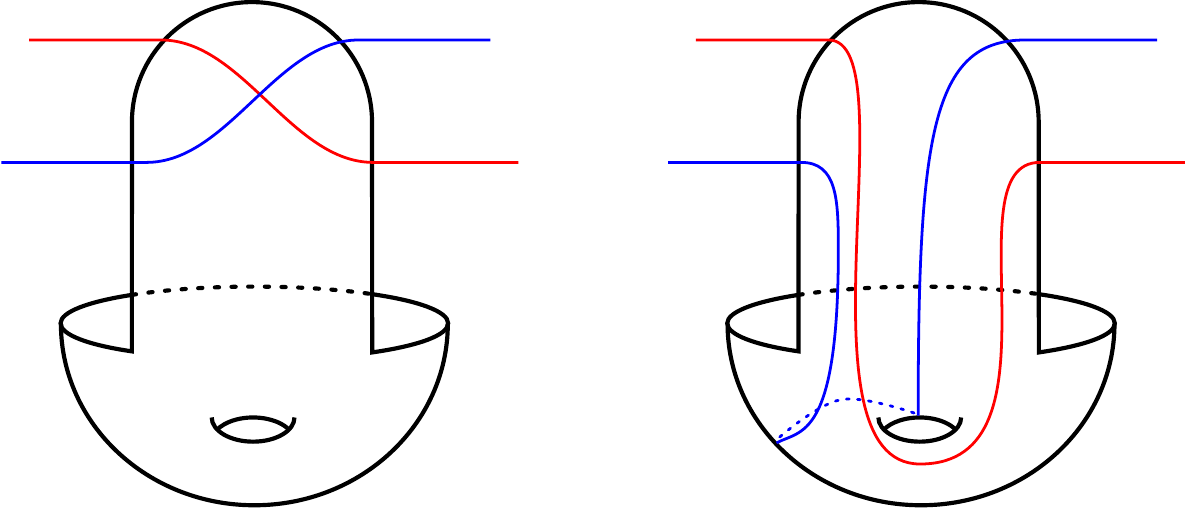}
\caption{De-crossing of $\gamma$ along $S$. The two local strands of $\gamma$ have been depicted in different colors for clarity.}
\label{figure:decross}
\end{figure}

In practice, the portion of $S$ connecting $p$ to the rest of $S$ can be quite long and thin. Where the clarity of a figure dictates, this will sometimes be depicted as an arc connecting $p$ to some genus $1$ subsurface.

Non-simple curves will arise as the image of simple curves under covering maps. Suppose $f: \Sigma \to \Sigma'$ is a regular covering of surfaces with deck group $G$. Let $p \in \Sigma'$ be given, and identify the fiber $f^{-1}(p)$ with the set $G\cdot p$. If $\gamma \subset \Sigma$ passes through points $g\cdot p, h\cdot p \in \Sigma$, then the image $f(\gamma)$ will have a double point at $p$. In this situation, we say that $\gamma$ has {\em local branches in sheets $g,h$}. The following lemma records some properties of the de-crossing procedure in this context.

\begin{lemma}\label{lemma:decross}
Let $f: \Sigma \to \Sigma'$ be a regular covering with deck group $G$. Suppose $\gamma \subset \Sigma$ is a simple closed curve; let $\overline{\gamma} \subset \Sigma'$ be the image $f(\gamma)$. Let $p$ be a double point of $\overline{\gamma}$, and let $S \subset \Sigma'$ be a genus $1$ subsurface disjoint from $\overline \gamma$ except in a neighborhood of $p$, and endowed with geometric symplectic basis $E, F$. Suppose that $f^{-1}(S)$ is a disjoint union of surfaces each homeomorphic to $S$. Define $\overline{\gamma}'$ to be the de-crossing of $\overline\gamma$ along $S$. Then the following assertions hold:
\begin{enumerate}
\item $\overline{\gamma}'$ lifts to a simple closed curve $\gamma' \subset \Sigma$.
\item Suppose the double point $p$ of $\overline{\gamma}$ arises from local branches of $\gamma$ in sheets $g, h$. Then in $H_1(\Sigma)$,
\[
[\gamma'] = [\gamma] + g \cdot E + h \cdot F.
\] 
\end{enumerate}
\end{lemma}
\begin{proof}
Both items will follow from an analysis of $f^{-1}(\overline{\gamma'})$ via the path-lifting construction. Choose a point $q \in \gamma$ not contained in $f^{-1}(S)$, and consider the component $\gamma^+$ of $f^{-1}(\overline{\gamma}')$ that passes through $q$. This lift will follow $\gamma$ until entering a component of $f^{-1}(S)$, where it follows some arc $\alpha$ into the interior of $f^{-1}(S)$, runs once around the preimage of $E$, then follows $f^{-1}(f(\alpha))$ back out of $f^{-1}(S)$ and rejoins the preimage of $\overline{\gamma}$. By the assumption that $E$ lifts to $\Sigma$, it follows that $\gamma^+$ rejoins $\gamma$ itself and not some other component of $f^{-1}(\overline{\gamma})$. The same analysis applies the second time that $\gamma$ passes over $p \in \sigma'$; this time, $\gamma^+$ looks locally like the curve-arc sum of $\gamma$ with some preimage of $F$. After passing through both points in $\gamma \cap f^{-1}(p)$, the lift $\gamma^+$ is still following $\gamma$ and not some other component of $f^{-1}(\overline{\gamma})$. Since $\overline{\gamma}$ and $\overline{\gamma}'$ coincide outside of $S$, it follows that $\gamma^+$ will follow $\gamma$ back to $q$, closing up as a simple closed curve as claimed. 
\end{proof}

\subsection{Exhibiting unipotents (I)} In this subsection and the next two we exhibit a large collection of elements of $\Gamma \cap \mathcal{U}$ and $\Gamma \cap \mathcal{U}^-$. As the arguments for $\mathcal{U}$ and $\mathcal U^-$ will be visibly identical, we will formulate our arguments only for the group $\mathcal U$.   

\begin{lemma}\label{lemma:exhibition1} $\pi(\Gamma \cap \mathcal U)$ contains a finite-index subgroup of $M_1$.
\end{lemma}
\begin{proof}
As an abelian group, $M_1$ is generated by the following set $\mathcal S$:
\[
\mathcal S := \{\Pi_{na} \xi v \mid \xi \in H,\ v \in \{E_i, F_i \mid 3 \le i \le g_0\} \}.
\]
To prove Lemma \ref{lemma:exhibition1}, it therefore suffices to produce, for each $v \in \mathcal S$, an element $T_v \in  \Gamma \cap \mathcal U$ such that $\pi(T_v) = nv$ for some $n \in \Z, n \ne 0$. 

\begin{claim}\label{claim:curves}
Fix $\xi \in H$ and $v \in \{E_i, F_i \mid 3 \le i \le g_0\} $ arbitrary. Then there exists a clean element $\gamma_{v, \xi} \in \pi_1(X)$ so that in the notation of Lemma \ref{lemma:clean}, there is some $k \in \Z/m\Z$ such that
\[
[\widetilde{(\gamma_{v, \xi})_L}] = E_2 + \zeta^k\xi\cdot v.
\]
\end{claim}
Modulo the claim, Lemma \ref{lemma:exhibition1} follows easily. Applying Lemma \ref{lemma:clean} to the element $\gamma_{v,\xi}$, it follows that for $x\in M$, 
\[
\rho(\gamma_{v,\xi}^m)(x) = x + \Pi_{na}\pair{x, E_2 + \zeta^k \xi v}[E_2 + \zeta^k \xi v]. 
\]
In particular, $\rho(\gamma_{v,\xi}^m) \in \mathcal P$. 

\begin{figure}[h]
\labellist
\small
\pinlabel $X^\circ$ [tl] at 133.6 17.6
\pinlabel $\gamma_L$ [tr] at 70.4 72
\pinlabel $\gamma_R$ [tl] at 95.2 67.2
\endlabellist
\includegraphics{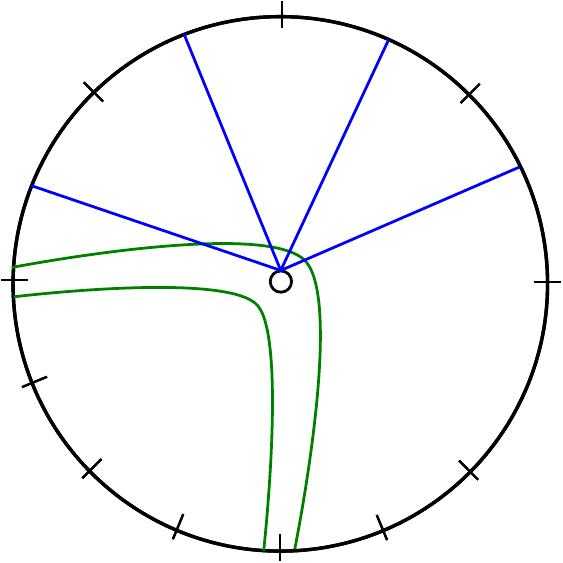}
\caption{The curves $\gamma_L, \gamma_R$ for the element $\gamma = [e_2,f_2] \in \pi_1(X)$. The branch cuts used in the construction of the cover $W \to X$ (i.e. the images of the arcs $\gamma_h, \gamma_v$ as in \eqref{equation:Ghdef}) are depicted in blue. For clarity, in this and the remaining figures, we suppress the edge identifications specified in Section \ref{subsection:repair}. Since neither $\gamma_L$ nor $\gamma_R$ cross $E_1, F_1$, both curves lift to a union of separating curves on $V^\circ$. As moreover no component of $s^{-1}(\gamma_L)$ crosses a branch cut, it follows that $\gamma$ is $W$-separating as desired.
}
\label{figure:Wsep}
\end{figure}

Figure \ref{figure:Wsep} shows that the element $[e_2,f_2] \in \pi_1(X)$ is clean and $W$-separating. Define
\[
R := \rho([e_2,f_2]).
\]
By Lemma \ref{lemma:Vsep}, $R(E_2) = E_2$ and $R(v) = \zeta^{-1} v$. Define
\[
T_{v,\xi} := [\rho(\gamma_{v,\xi}^m), R]. 
\]
Now applying Lemma \ref{lemma:commtrick}, it follows that
\[
\pi(T_{v,\xi}) =  \Pi_{na} (\ze^{-1}-1) \zeta^k \xi v.
\]
Thus $\pi(\Gamma \cap \mathcal U)$ contains all elements of the form $\Pi_{na} (\ze^{-1}-1) \zeta^k \xi v$ for $\xi \in H$ and $v \in \{E_i, F_i \mid 3 \le i \le g_0\}$ arbitrary. Letting $R \in \mathcal P$ act on $\pi(\Gamma \cap \mathcal U)$ by conjugation, Lemma \ref{lemma:parabolic} implies that 
\begin{equation}\label{equation:xiv}
\Pi_{na} (\ze^{-1}-1) \zeta^k \xi v \in \pi(\Gamma \cap \mathcal U)
\end{equation}
for $k$ arbitrary, hence simply $\Pi_{na} (\ze^{-1}-1) \xi v \in \pi(\Gamma \cap \mathcal U)$.

By construction, the action of $\zeta$ on $\Pi_{na} H_1(W)$ is fixed-point free. Hence the endomorphism $(\ze^{-1}-1)$ is invertible, and so the $\Q$-span of the vectors given in (\ref{equation:xiv}) is $\Q \otimes M_1$. Lemma \ref{lemma:exhibition1} follows, modulo Claim \ref{claim:curves}.

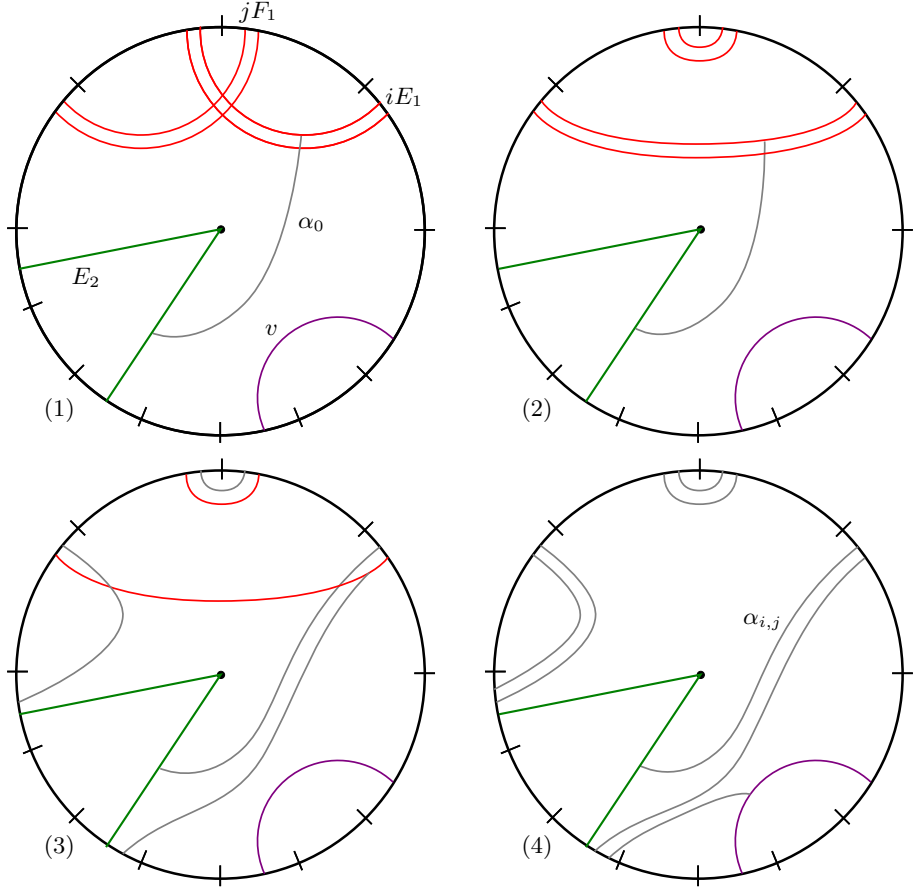
\begin{figure}[h]
\labellist
\small
\pinlabel $(1)$  at 20 180
\pinlabel $(2)$ at 200 180
\pinlabel $(3)$ at 20 15
\pinlabel $(4)$ at 200 15
\pinlabel $E_2$ at 30 230
\pinlabel $\alpha_0$ at 115 250
\pinlabel $\alpha_{i,j}$ at 285 100
\pinlabel $v$ at 100 210
\pinlabel $iE_1$ at 150 298
\pinlabel $jF_1$ at 95 330
\endlabellist
\includegraphics{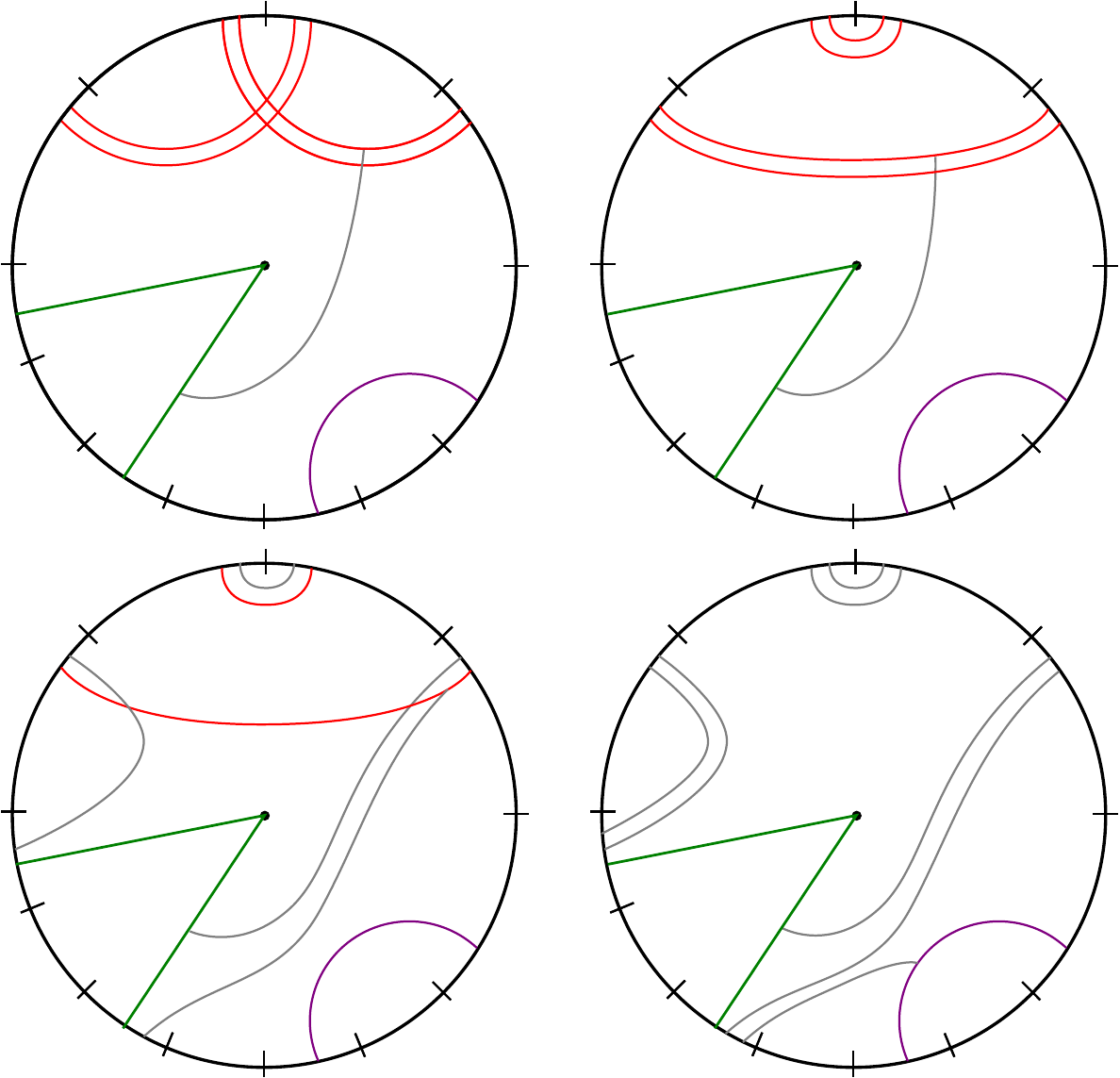}
\caption{The construction of $\gamma_{v, \xi}$, illustrated for $\xi = \sigma^2 \tau^2$. Panel 1: the constituent parts. Panel 2: First, resolve all intersections of the multicurve $iE_1 + jF_1$. Panel 3: To construct $\alpha_{i,j}$, run $\alpha_0$ along each component of $iE_1 + jE_2$, working from top to bottom. Panel 4: Finally, attach $\alpha_{i,j}$ to $v$.}
\label{figure:curves1}
\end{figure}

Claim \ref{claim:curves} is established using curve-arc sums. Figure \ref{figure:curves1} depicts an arc $\alpha_{i,j} \subset X$ connecting $E_2$ to an arbitrary element $v \in \{E_i, F_i \mid 3 \le i \le g_0\}$. For {\em any} element $\sigma^i \tau^j \in (\Z/m\Z)^2$, it is possible to construct an arc $\alpha_{i,j}$ such that the difference of the endpoints of the lift $s^{-1}(\alpha_{i,j})$ corresponds to the element $\sigma^i \tau^j$ (here we treat the sheets of the covering $s: V \to X$ as a torsor over $(\Z/m\Z)^2$). Let $\xi \in H$ have the form $\xi = \sigma^i \tau^j \zeta^k$. The curve $\gamma_{v, \xi}$ is then defined to be
\[
\gamma_{v,\xi} := E_2 +_{\alpha_{i,j}} v.
\]
By construction, each component of the preimage $s^{-1}(\gamma_{v,\xi})$ is curve-arc sum, with one particular component given by
\[
E_2 + _{s^{-1}(\alpha_{i,j})} \sigma^i \tau^j v.
\]
Moreover, the preimage $(s \circ t)^{-1}(\gamma_{v,\xi})$ is also a union of curve-arc sums, one of which is
\[
E_2 + _{(s \circ t)^{-1}(\alpha_{i,j})} \sigma^i \tau^j \zeta^k v,
\]
since the difference of the endpoints $(s \circ t)^{-1}(\alpha_{i,j})$ is $\sigma^i \tau^j \zeta^k$ for some $k$. The factor of $\zeta^k$ appearing above is not within our control, since the construction of $\alpha_{i,j}$ does not give any control over which sheet of the covering $t: W \to V$ the lift $t^{-1}(s^{-1}(\alpha_{i,j}))$ ends in. Claim \ref{claim:curves} now follows from Lemma \ref{lemma:casumfacts}.\end{proof}

\para{Exhibiting unipotents (II)} Lemma \ref{lemma:exhibition1} exhibits (multiples of) all elements of the form $E_i, F_i$ in $\pi( \Gamma \cap \mathcal U)$ for $i \ge 3$. We next build on this to show that $\pi( \Gamma \cap \mathcal U)$ contains multiples of elements of the form $E_1, F_1$. 

\begin{lemma}\label{lemma:exhibition2} $\pi(\Gamma \cap \mathcal U)$ contains a finite-index subgroup of $M_2$.  
\end{lemma}
\begin{proof}
As an abelian group, $M_2$ is generated by the set
\[
\mathcal S = \{\Pi_{na} \xi v \mid \xi \in H,\ v \in \{E_1, F_1\} \}.
\]
To prove Lemma \ref{lemma:exhibition2}, it therefore suffices to produce, for each $v \in \mathcal S$, an element $T_v \in  \Gamma \cap \mathcal U$ such that $\pi(T_v) = nv$ for some $n \in \Z, n \ne 0$. 

\begin{claim}\label{claim:e1f1}
There exist clean elements $\gamma_E, \gamma_F \in \pi_1(X)$ such that
\begin{align*}
[\widetilde{(\gamma_E)_L}] &= E_1 + \zeta\left (E_3 + m F_3- \sum_{i = 0}^{m-1} \sigma^i \zeta^{-i} F_3\right) =: E_1 + w_E,\\
[\widetilde{(\gamma_F)_L}] &= F_1 + \zeta^{-1}\left (E_3 + m F_3 - \sum_{i = 0}^{m-1} \tau^i F_3\right) =: F_1 + w_F.
\end{align*}

\end{claim}

We first see how Lemma \ref{lemma:exhibition2} follows from Claim \ref{claim:e1f1}. Since the arguments will be very similar, we will set $v \in \{E_1, F_1\}$ and suppress the subscript on $\gamma_E, \gamma_F$ in what follows. Applying Lemma \ref{lemma:clean} to $\gamma$ produces the element 
\[
\rho(\gamma_v^m)(x) = x + \Pi_{na}\pair{x,v+w}[v+w].
\]
Note in particular that $\rho(\gamma_v^m) \in \mathcal P$. Appealing to Lemma \ref{lemma:exhibition1}, for $\xi \in H$ arbitrary there is an element $T_\xi \in \Gamma \cap \mathcal U$ such that $\pi(T_\xi) = n \Pi_{na} \xi F_3$ for some $n \ne 0$. We record that
\[
\pair{\xi F_3, w} = \xi \zeta^{\pm},
\]
the value of the sign being determined by whether $w = w_E$ or $w=w_F$. Applying Lemma \ref{lemma:parabolic} to $\rho(\gamma_v^m) \in \mathcal P$ and $T_\xi \in \mathcal U$, we find
\begin{align*}
\pi(\rho(\gamma_v^m)\,T_\xi\,\rho(\gamma_v^m)^{-1} ) &= \rho(\gamma_v^m)(n \Pi_{na} \xi F_3)\\
	&= n \Pi_{na} \xi F_3 + \Pi_{na} \pair{n \Pi_{na} \xi F_3, v+w}[v+w]\\
	&= n \Pi_{na} \xi F_3 - n \Pi_{na}^2 \xi \zeta^{\pm}(v+w).
\end{align*}
By linearity and Lemma \ref{lemma:exhibition1}, this shows $n \Pi_{na}^2 \xi \zeta^{\pm} v \in \pi(\Gamma \cap \mathcal U)$. As $\Pi_{na}^2 = m \Pi_{na}$, Lemma \ref{lemma:exhibition2} follows, modulo Claim \ref{claim:e1f1}. 

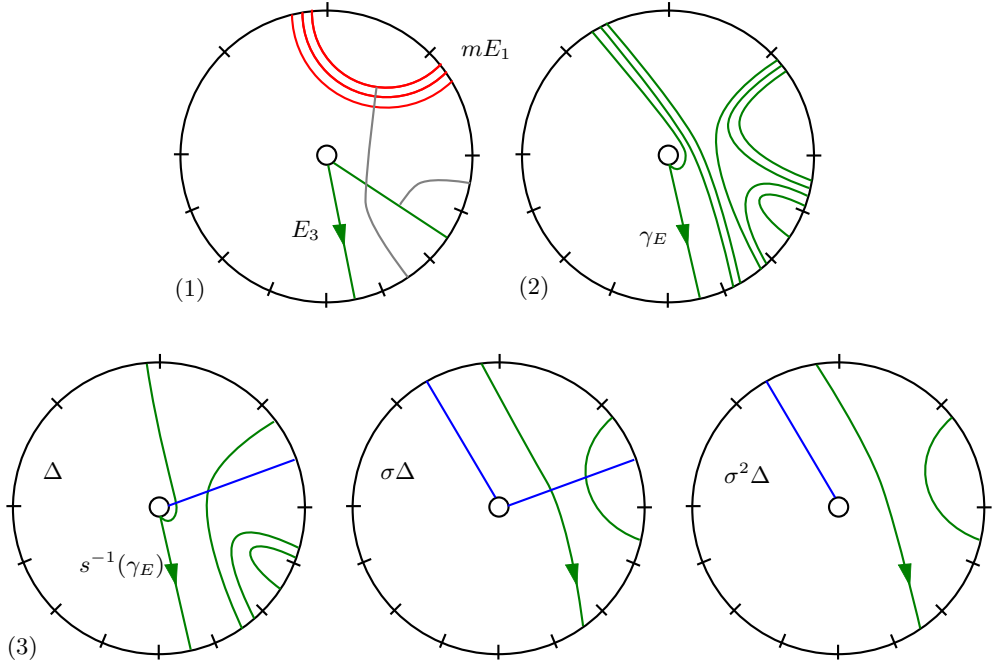
\begin{figure}[h]
\labellist
\small
\pinlabel (1) [tr] at 77.6 146.4
\pinlabel $E_3$ [tr] at 120.8 168
\pinlabel $mE_1$ [bl] at 169.6 224.8
\pinlabel (2) [tr] at 207.2 147.2
\pinlabel $\gamma_E$ [tr] at 252 164
\pinlabel (3) [tr] at 14.4 11.2
\pinlabel $\Delta$ [tl] at 12 76.8
\pinlabel $\sigma\Delta$ [tl] at 139.2 76.8
\pinlabel $\sigma^2\Delta$ [tl] at 268.8 76.8
\pinlabel $s^{-1}(\gamma_E)$ at 44 36
\endlabellist
\includegraphics{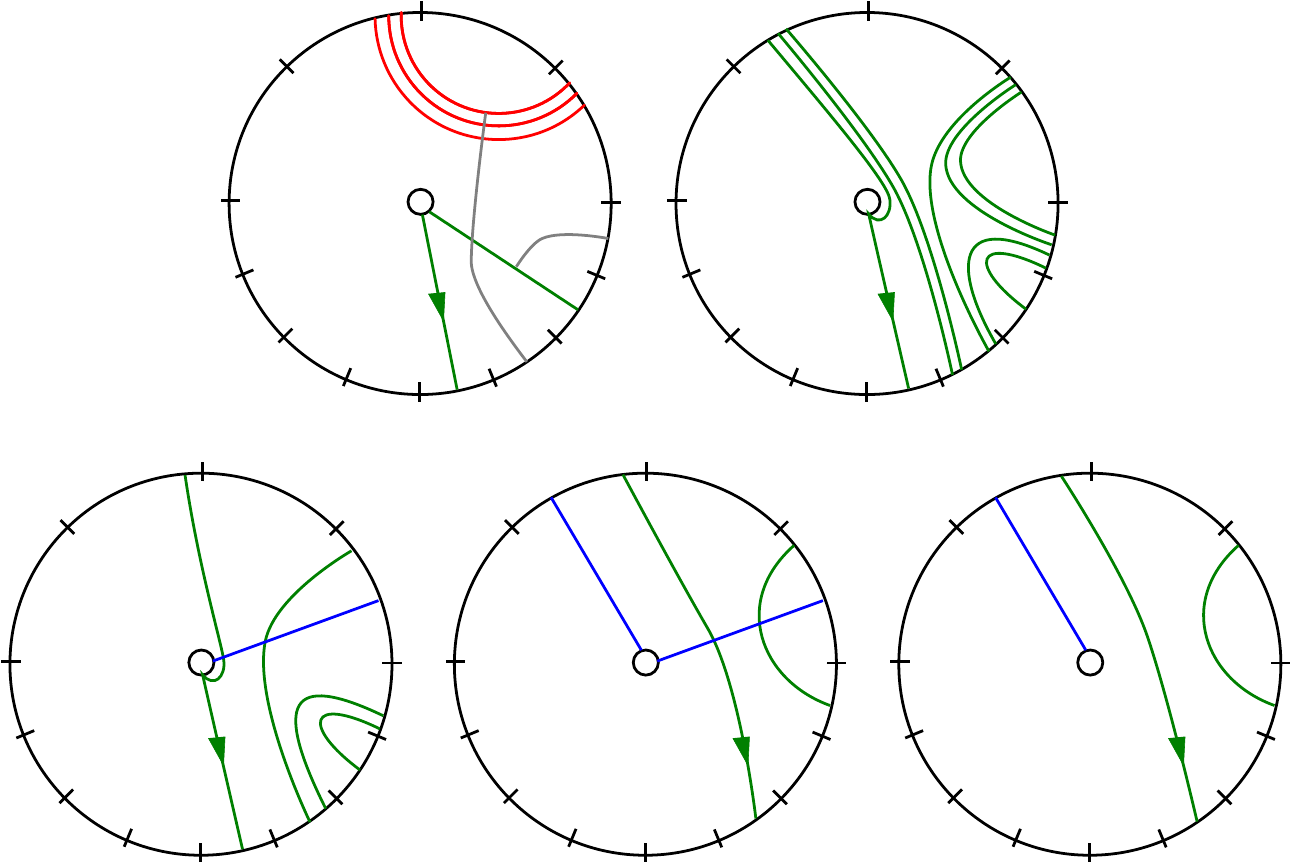}
\caption{The curve $\gamma_E$, illustrated for $m = 3$.}
\label{figure:E1}
\end{figure}
Claim \ref{claim:e1f1} is proved using similar techniques as in Claim \ref{claim:curves}, essentially by direct exhibition. Figure \ref{figure:E1} depicts the curve $\gamma_E$, illustrated there for $m=3$. Panel 1 shows how to build $\gamma_E$ as an iterated curve-arc sum of $E_3$ and $m$ copies of $E_1$. Panel 2 depicts the result of the construction, the curve $\gamma_E$. Panel 3 comprises the bottom half of the figure and consists of three sheets $\Delta, \sigma \Delta, \sigma^2 \Delta$ of the $9$-sheeted cover $V^\circ \to X^\circ$ (again for the case $m = 3$). The curve $\gamma_E$ has been lifted along the covering $s: V^\circ \to X^\circ$, where it remains a simple closed curve. The blue lines shown in Panel 3 indicate the branch cuts used in the construction of the covering $t: W^\circ \to V^\circ$.  In the sheets $\Delta, \sigma \Delta$, one sees $(s^{-1}(\gamma_E))_L$ crossing a branch cut twice, once in each direction. This shows that $(\gamma_E)_L$ lifts to $W^\circ$ as a simple closed curve, or equivalently, $\theta(s^{-1}((\gamma_E)_L)) = 0$. Altogether, Figure \ref{figure:E1} then shows that $\gamma_E$ is a clean element of $\pi_1(X)$. The determination of $[\widetilde{(\gamma_E)_L}] \in H_1(W)$ is a direct computation. One must remember to check that 
\[
\pair{\widetilde{(\gamma_E)_L}, G} = 0
\]
for any $G \in \Z[H]\pair{G_h, G_v}$, but this is easy: each such $G$ is either disjoint from $\widetilde{(\gamma_E)_L}$ or else crosses $\widetilde{(\gamma_E)_L}$ exactly twice with opposite signs.

\begin{figure}[h]
\labellist
\small
\endlabellist
\includegraphics{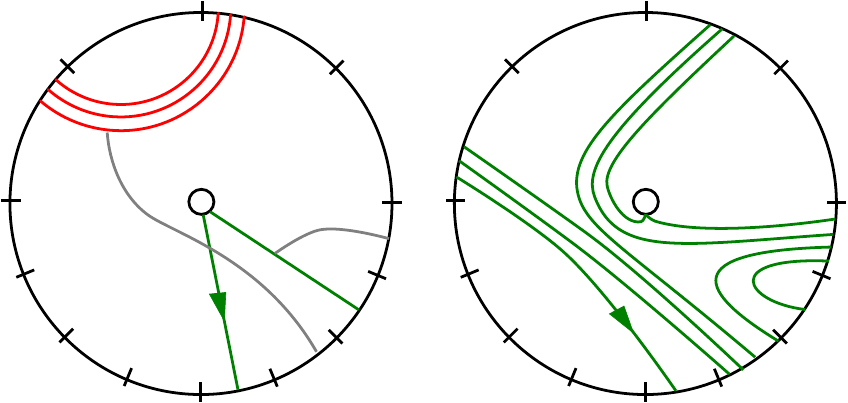}
\caption{The curve $\gamma_F$, illustrated for $m = 3$.}
\label{figure:F1}
\end{figure}

The construction of $\gamma_F$ proceeds along very similar lines. One performs an $m$-fold iterated curve-arc sum of $E_3$ and $F_1$ using the arc indicated in Figure \ref{figure:F1}. The rest of the argument then follows that for $\gamma_E$.
\end{proof}

\para{Exhibiting unipotents (III)} The final class of unipotents we must exhibit are supported on the summand $\Pi_{na} \Z[H]\pair{G_h, G_v} \le \Pi_{na}H_1(W;\Z)$. 

\begin{lemma}\label{lemma:exhibition3}
$\pi(\Gamma \cap \mathcal U)$ contains a finite-index subgroup of $M_3$. 
\end{lemma}

\begin{proof}
The proof follows the same outline as in Lemmas \ref{lemma:exhibition1} and \ref{lemma:exhibition2} -- a class of clean elements are exhibited and subsequently used to obtain the required elements in $M_3$.

\begin{claim}\label{claim:3}
There exist clean elements $\overline{G_h}', \overline{G_v}' \in \pi_1(X)$ such that in $\Pi_{na} H_1(W;\Z)$,
\begin{align*}
[\tilde{G_h}] &= [G_h] + \sum_{i = 3}^5 (\xi_i [E_i] + \eta_i [F_i])\\
[\tilde{G_v}] &= [G_v] + \sum_{i = 3}^5 (\xi_i [E_i] + \eta_i [F_i])
\end{align*}
for some elements $\xi_i, \eta_i \in H$. 
\end{claim}

Assuming Claim \ref{claim:3}, we prove Lemma \ref{lemma:exhibition3}. This follows the same principle as in Lemma \ref{lemma:exhibition2}. Once again the arguments for $G_h$ and $G_v$ are essentially identical, and the subscripts will be suppressed. Applying Lemma \ref{lemma:clean} to $\overline{G}'$, 
\[
\rho(\overline{G}'^m)(x) = x + \Pi_{na}\pair{x, G+\sum_{i = 3}^5 (\xi_i [E_i] + \eta_i [F_i])}[G+ \sum_{i = 3}^5 (\xi_i [E_i] + \eta_i [F_i])].
\]
In particular, $\rho(\overline{G}') \in \mathcal P$. Appealing to Lemma \ref{lemma:exhibition1}, for $\chi \in H$ arbitrary there is an element $T_\chi \in \Gamma \cap \mathcal U$ such that $\pi(T_\chi) = n \Pi_{na} \chi \xi_3 F_3$ for some $n \ne 0$. Observe that 
\[
\pair{\chi \xi_3 [F_3], \sum_{i = 3}^5 (\xi_i [E_i] + \eta_i [F_i])} = -\chi \xi_3 \xi_3^{-1} = - \chi.
\]
Applying Lemma \ref{lemma:parabolic} to $\rho(\overline{G}'^m) \in \mathcal P$ and $T_\chi \in \mathcal U$, we find
\begin{align*}
\pi(\rho(\overline{G}'^m) T_\chi \rho(\overline{G}'^m)^{-1}) &= \rho(\overline{G}'^m)(n \Pi_{na} \chi \xi_3 F_3)\\
&= n \Pi_{na} \chi \xi_3 F_3 + \Pi_{na}\pair{n \Pi_{na} \chi \xi_3 F_3, G + \sum_{i = 3}^5 (\xi_i E_i + \eta_i F_i)}[G + \sum_{i = 3}^5 (\xi_i E_i + \eta_i F_i)]\\
&= n \Pi_{na} \chi \xi_3 F_3 - \chi [G + \sum_{i = 3}^5 (\xi_i E_i + \eta_i F_i)].
\end{align*}
Lemma \ref{lemma:exhibition3} now follows as in the proof of Lemma \ref{lemma:exhibition2}. 

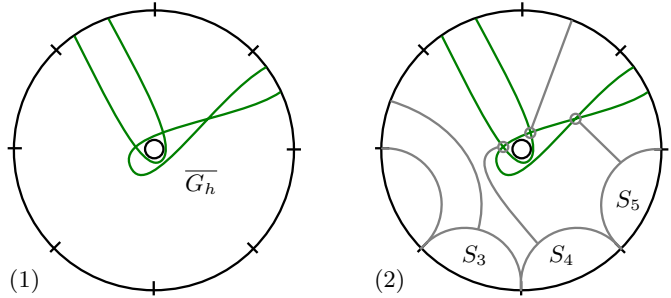
\begin{figure}[h]
\labellist
\small
\pinlabel (1) [tr] at 14.4 12
\pinlabel $\overline{G_h}$ [tl] at 64.8 49.6
\pinlabel (2) [tr] at 152 12
\pinlabel $S_3$ [c] at 176 15
\pinlabel $S_4$ [c] at 208.8 15.2
\pinlabel $S_5$ [tl] at 227 41.6
\endlabellist
\includegraphics{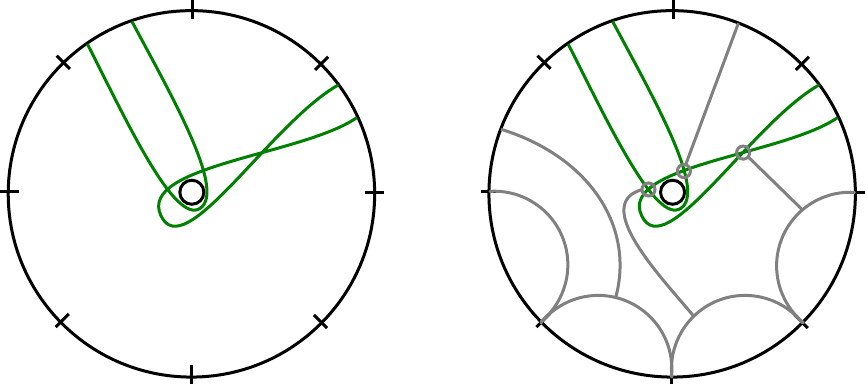}
\caption{Resolving the double points of $\overline{G_h}$ by de-crossing.}
\label{figure:G}
\end{figure}
We proceed to the proof of Claim \ref{claim:3}. The first panel of Figure \ref{figure:G} shows the image $\overline{G_h}$ of $G_h$ in $X^\circ$. To obtain this figure, we have perturbed the original curve $G_h:= (1-\zeta) \gamma_h$ so that it does not pass through the branch locus of $W$, and then projected via $(s \circ t): W^\circ \to X^\circ$. As depicted, $\overline{G_h}$ has three double points which we wish to resolve. This can be accomplished via the de-crossing procedure, shown in panel 2. The assumption $g_0 \ge 5$ is used here to ensure the existence of three disjoint subsurfaces $S_3, S_4, S_5$ of genus $1$, each disjoint from the curves $E_1, F_1, E_2, F_2$, and each satisfying the hypotheses of Lemma \ref{lemma:decross}. The result of the de-crossing is a simple curve $\overline{G_h}' \subset X^\circ$. We can convert $\overline{G_h}'$ into a simple based loop by attaching $\overline{G_h'}$ to the basepoint $x_0 \in X$.

By Lemma \ref{lemma:decross}.1, $\overline{G_h}'$ lifts to $W^\circ$ and so determines a clean element of $\pi_1(X)$. Moreover, Lemma \ref{lemma:decross}.2 asserts that the lift $\tilde{G_h}$ to $W$ satisfies
\[
[\tilde{G_h}] = [G_h] + \sum_{i = 3}^5 (\xi_i [E_i] + \eta_i [F_i])
\]
for some elements $\xi_i, \eta_i \in H$. The argument for $G_v$ is virtually identical. 
\end{proof}

\noindent \textit{Proof of Theorem \ref{theorem:main}:} Lemma \ref{lemma:Mfree} shows that $M= \Pi_{na}H_1(W;\Q)$ is a free $A = \Pi_{na}\Q[H]$-module, and Lemma \ref{lemma:ourUU} shows that the elements $E_2,F_2, E_3,F_3$ satisfy the hypotheses required of the elements $x_1, x_1^*, x_2, x_2^*$. Lemmas \ref{lemma:ourUbar}, \ref{lemma:exhibition1}, \ref{lemma:exhibition2}, \ref{lemma:exhibition3} combine to show that $\Gamma$ contains enough unipotents in the sense of Proposition \ref{proposition:UU}. Theorem \ref{theorem:main} now follows from Proposition \ref{proposition:UU}. \qed

\section{Monodromy of the classical Atiyah--Kodaira manifolds}\label{section:nonnormal}

In Section \ref{subsection:nn} we gave a construction of the ``classical'' Atiyah--Kodaira manifolds $E^{nn}(X,m)$. Recall that the fiber of $E^{nn}(X,m)$ is the surface $Z$, an intermediate cover of $W \to X$. In this section, we return to the problem of computing the monodromy group $\Ga^{nn}(X,m)$ of these classical Atiyah--Kodaira manifolds, and we prove Theorems \ref{theorem:main-nn} and \ref{theorem:zariski-nn}. 

Fix $X$ and $m$ and denote $\Ga^{nn}=\Ga^{nn}(X,m)$.  The subgroup $\Gamma^{nn}\le \Aut(H_1(Z;\Q), (\cdot, \cdot))$ is centralized by the covering group $Q \cong \Z/m\Z$ for the branched covering $Z \to Y$. To understand this constraint, we first recall the representation theory of $\Q[Q]\simeq\Q[\Z/m\Z]$. 

\begin{proposition}[Representations of $\Z/m\Z$]
Let $Q=\pair\ze\simeq\Z/m\Z$. For each $k\mid m$, there is a unique (isomorphism class of) simple $\Q[Q]$-module where $\ze$ acts with order $k$. This module can be identified with the cyclotomic field $\Q(\ze_k)$ with $\ze\in Q$ acting by multiplication by $\ze_k=e^{2\pi i/k}$. Consequently, there is a Wedderburn decomposition $\Q[Q]=\prod_{k\mid m}\Q(\ze_k)$. 
\end{proposition} 

Then we decompose $H_1(Z;\Q)\simeq \bigoplus_{k\mid m}N_k$ into isotypic factors, and via Lemma \ref{lemma:rhotarget}, we have\[\Gamma^{nn} \le \prod_{k\mid m} \Aut_Q(N_k, \pair{\cdot, \cdot}_Q). \]
We denote $\mbf G_k=\Aut_Q(N_k,\pair{\cdot,\cdot}_Q)$. 

\begin{lemma}
The projection of $\Ga^{nn}$ to $\mbf G_1$ is trivial. 
\end{lemma}
\begin{proof}
This is similar to Lemma \ref{lemma:gammaimage}. The claim is equivalent to showing that $\Ga^{nn}$ acts trivially on $N_1=H_1(Z;\Q)^{\pair\ze}$. By transfer $H_1(Z;\Q)^{\pair\ze}\simeq H_1(Y;\Q)$, so the action of $\Ga^{nn}$ on $N_1$ is by the monodromy of the bundle $B\times Y\ra B$. The monodromy of this bundle is by point-pushing homeomorphisms, which act trivially on $H_1(Y)$. 
\end{proof}

This establishes the ``obvious upper bound" $\mbf G^{nn}<\prod_{k\mid m,\>k\neq 1}\mbf G_k$ on the Zariski closure $\mbf G^{nn}$  of $\Ga^{nn}$ mentioned in the introduction. Our first aim will be to compute $\mbf G^{nn}$ precisely using Theorem \ref{theorem:main}.

\para{Relating ``non-normalized" and ``normalized" monodromies} We explain how $\mbf G^{nn}$ can be computed from the Zariski closure $\mbf G$ of the monodromy group $\Ga=\Ga(X,m)$ of the normalized Atiyah--Kodaira bundle $E(X,m)\ra B'$. The proof of Theorem \ref{theorem:main-nn} will follow from this analysis. 
To begin, observe that the cover $W \to Z$ is regular with covering group $\pair{\tau}$. From this point of view, 
\[
Q \cong \pair{\tau, \zeta} / \pair{\tau} \cong \pair{\zeta} \cong \Z/m\Z.
\]
By the transfer homomorphism,
\[H_1(Z;\Q) \cong H_1(W;\Q)^{\pair{\tau}}.\]
Next we compare the module structures on $H_1(W;\Q)$ and $H_1(Z;\Q)$. Let $H_1(W;\Q)=\bigoplus_{k\mid m}M_k$ be the decomposition into $\Q[Q]$-isotypic factors. Since $Q<H$ is central, each $M_k$ is a $\Q[H]$ module. We denote $\mbf G_k=\Aut_H(M_k,\pair{\cdot,\cdot}_H)$ and we denote the image of $\Ga$ in $\mbf G_k$ by $\Ga_k$. 

\begin{lemma}\label{lemma:nn-factors}
Fix $k\mid m$. Taking $\tau$-invariants $M_k\leadsto M_k^{\pair\tau}$ induces a homomorphism $\alpha_k:\mbf G_k\ra\mbf G_k^{nn}$. Furthermore, $\alpha_k(\Ga_k)$ is a subgroup of finite index in $\Ga^{nn}_k$. 
\end{lemma}

\begin{proof}
By transfer, $N_k=M_k^{\pair \tau}$. If $f\in \Aut_H(M_{k},\pair{\cdot,\cdot}_H)$, then $f$ commutes with each $h\in H$, and in particular with $\tau$ and $\ze$, so $f$ preserves $M_{k}^{\pair\tau}$ and commutes with the $Q$-action on $M_{k}^{\pair\tau}$. To show the Reidemeister pairing is preserved, choose a set $T\subset H$ of coset representatives for $Q\backslash H$ with $1\in T$. For $x,y\in M_k$ we write
\[\pair{x,y}_H=\sum_{t\in T}\sum_{q\in Q}(x,qty)qt=\sum_{t\in T}\pair{x,ty}_Qt.\]
Then $\pair{fx,fy}_H=\pair{x,y}$ implies that $\pair{fx,fy}_Q =\pair{x,y}$, since $T\subset\Q[H]$, as a set of coset representatives, is linearly independent over $\Q[Q]$. 

To see that $\alpha_k(\Ga_k)<\Ga_k^{nn}$ is finite index, consider the following commutative diagram (with the notation from Section \ref{section:AKconstruction}). 
\[\begin{xy}
(-10,0)*+{\pi_1(B')}="A";
(12,0)*+{\Ga_k}="B";
(-10,-12)*+{\pi_1(B)}="C";
(12,-12)*+{\Ga^{nn}_k}="D";
{\ar@{->>}"A";"B"}?*!/_3mm/{};
{\ar@{^{(}->} "A";"C"}?*!/^5mm/{};
{\ar "B";"D"}?*!/_3mm/{\alpha_k};
{\ar@{->>}"C";"D"}?*!/_3mm/{};
\end{xy}\]
The horizontal maps are surjective by definition. Since $\pi_1(B')<\pi_1(B)$ is finite index, so too is $\alpha_k(\Ga_k)<\Ga_k^{nn}$. 
\end{proof}

Next we determine the kernel of $\al_k$. Since $M_k$ is a $\Q[H]$-module, it decomposes further into $H$-isotypic factors $M_k=\bigoplus_{\chi\in\bar I_k} M_{k,\chi}$ as in Section \ref{section:heisenberg}. We decompose $\bar I_k=\bar I_k'\sqcup \bar I_k''$ according to whether the $\tau$-invariant subspace of the corresponding simple $\Q[H]$-module is trivial or nontrivial, respectively. Denoting $\mbf G_{k,\chi}=\Aut_H(M_{k,\chi},\pair{\cdot,\cdot}_H)$, we have $\mbf G_k=\prod_{\chi\in\bar I_k}\mbf G_{k,\chi}$ and also $\mbf G_k=\prod_{\bar I_k'}\mbf G_{k,\chi}\times\prod_{\bar I_k''}\mbf G_{k,\chi}$.

\begin{lemma}\label{lemma:tauinvariants}
The kernel of $ \alpha_k$ is $\prod_{\bar I_k'}\mbf G_{k,\chi}$. 
\end{lemma}
\begin{proof}
It is clear that $\prod_{\bar I_k'}\mbf G_{k,\chi}<\ker\alpha_k$. The fact that the kernel is not larger follows from inspection of the irreducible representations of $H$. For each, the $\tau$-eigenspaces are permuted transitively by $\si\in H$. It follows that if $f\in \Aut_H(M_k,\pair{\cdot,\cdot}_H)$ acts trivially on $N_k=\bigoplus_{I_k''} M_{k,\chi}^{\pair\tau}$, then $f$ acts trivially on $\bigoplus_{I_k''}M_{k,\chi}$. 
\end{proof}

In summary, there is a homomorphism $\alpha:\mbf G\ra\mbf G^{nn}$ such that (i) the image of $\alpha$ is isomorphic to $\prod_{k\mid m,\>k>1}\prod_{\bar I_k''}\mbf G_{k,\chi}$, and (ii) the group $\alpha(\Ga)$ is of finite index in $\Ga^{nn}$. It follows that $\mbf G^{nn}\simeq\prod_{k\mid m,\>k>1}\prod_{\bar I_k''}\mbf G_{k,\chi}$. Furthermore, by Theorem \ref{theorem:main}, $\alpha(\Ga)<\mbf G^{nn}$ is arithmetic. Thus $\Ga^{nn}$ is also arithmetic. This establishes Theorem \ref{theorem:main-nn}. 

\para{Comparing $\mbf G_k$ and $\mbf G_k^{nn}$} Next we explain Theorem \ref{theorem:zariski-nn}, which amounts to showing that $\mbf G_k$ and $\mbf G_k^{nn}$ are isomorphic if and only if $k=m$. 

The group $\mbf G_k^{nn}=\Aut_Q(N_k,\pair{\cdot,\cdot}_Q)$ is an algebraic $\Q(\ze_k)^+$-group, where $\Q(\ze_k)^+<\Q(\ze_k)$ is the maximal real subfield. The module $N_k$ is a vector space over $\Q(\ze_k)$ of dimension $md$, where $d=2g_0-1$ and $g_0$ is the genus of $X$. Choose an isomorphism $N_k\simeq\Q(\ze_k)^{md}$. The matrix $B\in M_{md}(\Q(\ze_k))$ for $\pair{\cdot,\cdot}_Q$ with respect to the standard basis of $\Q(\ze_k)^{md}$ is skew-Hermitian with respect to the involution $\ze_k\mapsto\bar \ze_k$. Therefore, 
\[\mbf G_k^{nn}\simeq\{g\in M_{md}(\Q(\ze_k)): g^tB\bar g=B\}.\]
When $k=1$ or $2$, the involution on $\Q(\ze_k)$ is trivial, and $\mbf G^{nn}_k$ is a symplectic group over $\Q$. If $k>2$, then $\mbf G^{nn}_k$ is a unitary group. This situation is similar to what is detailed in \cite{looijenga}. In any case, $\mbf G_k^{nn}$ is an absolutely almost simple algebraic $\Q(\ze_k)^+$-group \cite[\S18.5]{witte-morris}.  

There is a similar description for $\mbf G_k$. Recall that $\mbf G_k=\prod\mbf G_{k,\chi}$, where $\mbf G_{k,\chi}=\Aut_H(M_{k,\chi},\pair{\cdot,\cdot}_H)$. The Reidemeister pairing restricted to $M_{k,\chi}$ takes values in $M_k(\Q(\chi))$ (the corresponding factor in the Wedderburn decomposition of $\Q[H]$). 
According to Lemma \ref{lemma:Mfree}, $M_{k,\chi}\simeq M_k(\Q(\chi))^d$ where $d=2g_0-1$. After choosing a basis, we express $\pair{x,y}_H=\bar x^tCy$ for $x,y\in M_{k,\chi}$, where $C\in M_d(M_k(\Q(\chi)))\simeq M_{kd}(\Q(\chi))$ is a skew-Hermitian matrix. Then
\[\mbf G_{k,\chi}\simeq \{g\in M_{kd}(\Q(\chi)): g^tC\bar g=C\}\] 

In order for $\prod_{\bar I_k''}\mbf G_{k,\chi}<\mbf G_k^{nn}$ to be finite index, we must have $|\bar I_k''|=1$. This is because $\mbf G_k^{nn}$ is almost simple, and if $|\bar I_k''|>1$, then $\prod_{\bar I_k''}\mbf G_{k,\chi}$ is not almost simple. According to Proposition \ref{proposition:QH-tau-invariants}, if $k>1$, then $|I_k''|=1$ if and only if $k=m$. In this case, we show

\begin{proposition}\label{proposition:top-component}
The homomorphism $\al_m:\mbf G_m\ra\mbf G_m^{nn}$ is an isomorphism. 
\end{proposition} 

\begin{proof}
Recall that $M_m\sbs H_1(W;\Q)$ and $N_m\sbs H_1(Z;\Q)\simeq H_1(W;\Q)^{\pair\tau}$ are the subspaces where $\ze\in Q<H$ acts with order $m$. From Lemma \ref{lemma:Mfree}, we know $M_m\simeq A_m^d$, where $A_m\simeq M_m(\Q(\ze_m))$. Henceforth we will drop the subscript $m$ and simply write $M,N,A$. To prove the proposition, we will compare the forms $\pair{\cdot,\cdot}_H:M\times M\ra A$ and $\pair{\cdot,\cdot}_Q:N\times N\ra\Q(\ze_m)$. Once we show that these forms define the same algebraic group, it will follow that $\al_m$ is an isomorphism. 

First we describe the form $\pair{\cdot,\cdot}_H:M\times M\ra A$. There is a basis $\{E_2,F_2,\ldots,E_{g_0},F_{g_0},x\}$ for $M$ over $A$, where the vector $x$ is in the submodule spanned by $E_1,F_1$ and $G_v,G_h$ (technically, we mean to take the projection of $E_i,F_i$ to $M$, since $E_i,F_i\notin M$). With respect to this basis, $\pair{\cdot,\cdot}_H$ has matrix 
\[C=\left(
\begin{array}{cccc}
0&I_{(d-1)/2}&0\\
-I_{(d-1)/2}&0&0\\
0&0&\pair{x,x}_H
\end{array}
\right)\in M_d(A).\]
In what follows, it will be helpful to understand $\pair{x,x}_H\in A$ via the isomorphism $A\simeq M_m(\Q(\ze_m))$. 

Let $\rho:\Q[H]\ra M_m(\Q(\ze_m))$ be the surjection of (\ref{equation:heisenberg-rep}) with $a=b=0$ and $c=1$. This surjection splits via the map $\rho(h)\mapsto h\cdot e$, where $e=1-\frac{1}{m}(1+\ze+\cdots+\ze^{m-1})$ (thus $e$ is a primitive central idempotent). Using this, in what follows we will 
conflate a matrix in $M_m(\Q(\ze_m))$ with the corresponding element of $\Q[H]$. Let $E_{ij}\in M_m(\Q(\ze_m))$ denote the matrix with 1 in the $(i,j)$-entry and zeros elsewhere. For $1\le i, j\le m$, observe that 
\[E_{ij}=\frac{1}{m}\sum_{\ell=0}^{m-1}(\ze^{1-i}\tau)^\ell\si^{i-j} e.\]
We write $\pair{x,x}_H=\pair{x,x}_He=\sum_{h\in H}(x,hx)he$. By writing each $he$ as a sum of matrix coefficients (e.g.\ $\tau e=\sum_i \ze_m^{i-1}E_{ii}$), we can write $\pair{x,x}_H=\sum \pair{x,x}_{H;ij}\>E_{ij}$, where 
\[\pair{x,x}_{H;ij}=m\sum_{k=0}^{m-1} (x,\ze^k E_{ij}x)\ze_m^{-k}.\]
This expression gives the entries of the matrix $\pair{x,x}_H\in A\simeq M_m(\Q(\ze_m))$. 

Next we compare the matrix $C$ with the matrix for $\pair{\cdot,\cdot}_Q:N\times N\ra \Q(\ze_m)$. Recall $N\simeq M^{\pair\tau}\simeq (A^d)^{\pair\tau}\simeq (A^{\pair\tau})^d$. Here $\tau$ acts on $A$ by left multiplication by $\rho(\tau)$, so $A^{\pair\tau}\simeq\Q(\ze_m)^m$ is generated by $E_{11}\si^i$ for $0\le i\le m-1$. Then the basis $\{E_2,F_2,\ldots,E_{g_0},F_{g_0},x\}$ for $M$ gives a basis for $N$, and with respect to this basis, the form $\pair{\cdot,\cdot}_Q$ has matrix with blocks of the following form
\[B=\left(
\begin{array}{cccc}
0&I_{m(d-1)/2}&0\\
-I_{m(d-1)/2}&0&0\\
0&0&\beta
\end{array}
\right)\in M_{md}(\Q(\ze_m)).\]
Here $\beta\in M_m(\Q(\ze_m))$ is the matrix $\beta_{ij}=\pair{E_{11}\si^{1-i}x,E_{11}\si^{1-j}x}_Q$. One computes (recalling that $\tau E_{11}=\tau\cdot\frac{1}{m}\sum_{\ell=0}^{m-1}\tau^\ell=E_{11}$) that 
\[
\begin{array}{rclcl}
\beta_{ij}&=&\sum_{k=0}^{m-1}(E_{11}\si^{1-i}x,\ze^kE_{11}\si^{1-j}x)\ze_m^{-k}\\[1.5mm]
&=&\sum_{k=0}^{m-1}\frac{1}{m}\sum_{\ell=0}^{m-1}(\tau^\ell\si^{1-i} x,\ze^kE_{11}\si^{1-j}x)\ze_m^{-k}\\[1.5mm]
&=&\sum_{k=0}^{m-1}(x,\ze^k\si^{i-1}E_{11}\si^{1-j}x)\ze_m^{-k} \\[1.5mm]
&=&\sum_{k=0}^{m-1}\frac{1}{m}(x,\ze^k \sum_{\ell=0}^{m-1} (\ze^{1-i}\tau)^\ell \si^{i-j}x)\ze_m^{-k}\\[1.5mm]
&=&\sum_{k=0}^{m-1}(x,\ze^kE_{ij}x)\ze^{-k}_m\\[1.5mm]
&=&\frac{1}{m}\pair{x,x}_{H;ij}
\end{array}\]
Note that $\ze_m^{-k}$ appears rather than $\ze_m^k$ because $\ze\in Q<H$ acts by $\ze_m^{-1}$ on $N=M^{\pair\tau}$. From the above computation, we conclude that $B,C\in M_{md}(\Q(\ze_m))$ define the same unitary group, so $\mbf G_m\simeq\mbf G_m^{nn}$. 
\end{proof}
This finishes the proof of Theorem \ref{theorem:zariski-nn}. To end this section, we give an example that illustrates the case of Theorem \ref{theorem:zariski-nn} when $m$ is composite. 

\begin{example}Take $m=4$. Here the centralizer $\Sp_{2g}(\Q)^Q$ is isomorphic to $\Sp_{2g'}(\Q)\times \Sp_{2g'+2}(\Q)\times\SU(g',g'+2;\Q(i))$, where $g'$ is the genus of $Y=Z/Q$ (in terms of $g_0=\text{genus}(X)$, $g'=4g_0-3$). In this case $\mbf G^{nn}$ is isomorphic to 
\[\Sp_{g'+1}(\Q)\times\Sp_{g'+1}(\Q)\times\SU(g',g'+2;\Q(i))< \Sp_{2g'+2}(\Q)\times\SU(g',g'+2;\Q(i)).\]
\end{example}

%

\bibliographystyle{alpha}
\bibliography{AKarith.bib}

\end{document}